\documentclass[reqno, 11pt, a4paper]{amsart}

\usepackage[margin=1in]{geometry}
\usepackage{amssymb, amsthm, amsfonts, amscd}

\usepackage{xpatch}
\xpatchcmd{\proof}
  {\itshape}
  {\bfseries}
  {}
  {}
\DeclareMathOperator*{\motimes}{\text{\raisebox{0.4ex}{\scalebox{0.6}{$\bigotimes$}}}}
\DeclareMathOperator*{\moplus}{\text{\raisebox{0.4ex}{\scalebox{0.6}{$\bigoplus$}}}}

\usepackage{graphicx}  
\usepackage[all]{xy}
\usepackage[usenames,dvipsnames]{xcolor}
\usepackage{mathrsfs}
\usepackage{tikz-cd} 
\usepackage{stmaryrd}
\usepackage{verbatim}
\usepackage{blkarray}

\usepackage{tabularx}
\usepackage{latexsym}
\usepackage{todonotes}
\usepackage[shortlabels]{enumitem} 
\usepackage{bbm}
\usepackage{pdfpages}
\usepackage{titletoc}
\usepackage{booktabs}
\usepackage{mathtools}
\usepackage{thmtools}
\usepackage{hhline}
\usepackage{caption}
\usepackage{multirow, url}
\usepackage{textgreek}
\usepackage{upgreek}
\usepackage{textcomp}
\DeclareMathAlphabet{\cmcal}{OMS}{cmsy}{m}{n}

\usepackage{arydshln}

\usepackage{libertine}
\usepackage[LGR, T1]{fontenc}
\usepackage[utf8]{inputenc}

\definecolor{blue1}{rgb}{0, 0, 2}
\definecolor{sky}{rgb}{0, 0.2, 0.8}
\usepackage[colorlinks=true, citecolor=sky, linkcolor=magenta, urlcolor=blue1, filecolor=cyan]{hyperref}

\newtheorem{theorem}{Theorem}[section]

\newtheorem{lemma}[theorem]{Lemma}
\newtheorem{proposition}[theorem]{Proposition}

\newtheorem{conjecture}[theorem]{Conjecture}

\theoremstyle{definition}
\newtheorem{definition}[theorem]{Definition}

\theoremstyle{remark}
\newtheorem{remark}[theorem]{\bf{Remark}}

\numberwithin{equation}{section} \numberwithin{table}{subsection}
	
\newtheorem*{theorem*}{\bf{Theorem}}
\newtheorem*{claim*}{\bf{Claim}}
\newtheorem*{remark*}{\bf{Remark}}
\newtheorem*{remarks*}{\bf{Remarks}}
\newtheorem*{example*}{\bf{Example}}
\newtheorem*{examples*}{\bf{Examples}}

\newcommand{\C}{{\mathbf{C}}}

\newcommand{\F}{{\mathbf{F}}}

\newcommand{\bP}{{\mathbf{P}}}
\newcommand{\Q}{{\mathbf{Q}}}

\newcommand{\T}{{\mathbf{T}}}

\newcommand{\Z}{{\mathbf{Z}}}

\newcommand{\fh}{{\mathfrak{h}}}

\newcommand{\fm}{{\mathfrak{m}}}
\newcommand{\fn}{{\mathfrak{n}}}

\newcommand{\fI}{{\mathfrak{I}}}

\newcommand{\cA}{{\cmcal{A}}}
\newcommand{\cB}{{\cmcal{B}}}

\newcommand{\cH}{{\cmcal{H}}}

\newcommand{\cJ}{{\cmcal{J}}}

\newcommand{\cS}{{\cmcal{S}}}
\newcommand{\cT}{{\cmcal{T}}}

\newcommand{\scC}{{\mathscr{C}}}

\newcommand{\bbO}{\mathbb{O}}

\newcommand{\tn}{\textnormal}
\newcommand{\dmid}{\mathrel\Vert}

\def\a{\alpha}
\def\b{\beta}

\def\d{\delta}

\def\ve{\varepsilon}

\def\p{\varphi}

\def\G{\Gamma}

\def\<{\left\langle}
\def\>{\right\rangle}
\def\gcd{\textnormal{gcd}}

\def\min{\textnormal{min}}
\def\dim{\textnormal{dim}}

\newcommand{\zmod}[1]{{\Z/{#1}\Z}}

\newcommand{\inj}{\hookrightarrow}
\newcommand{\surj}{\twoheadrightarrow}
\newcommand{\arinj}{\ar@{^(->}}
\newcommand{\arsurj}{\ar@{->>}}
\newcommand{\arsub}{\ar@{}[r]|-*[@]{\subset}}
\newcommand{\arsup}{\ar@{}[r]|-*[@]{\supset}}
\newcommand{\arcap}{\ar@{}[d]|-*[@]{\subset}}
\newcommand{\arcup}{\ar@{}[u]|-*[@]{\subset}}
\newcommand{\arin}{\ar@{}[u]|-*[@]{\in}}

\renewcommand{\pmod}[1]{{\,(\textnormal{mod}\hspace{1mm} {#1})}}
\renewcommand{\mod}[1]{{~~\textnormal{mod}\hspace{1mm} {#1}}}

\newcommand{\Gal}{{\textnormal{Gal}}}
\newcommand{\Div}{{\textnormal{Div}}}
\newcommand{\Pic}{{\textnormal{Pic}}}

\newcommand{\End}{{\textnormal{End}}}

\newcommand{\SL}{{\textnormal{SL}}}

\newcommand{\reg}{\textnormal{reg}}

\newcommand{\sHom}{{\mathscr{H}\kern-.5pt om}}
\newcommand{\sExt}{{\mathscr{E}\kern-.5pt xt}}

\newcommand{\tor}{{\textnormal{tors}}}

\newcommand{\val}{\textnormal{val}}

\renewcommand{~}{\hspace*{0.5mm}}

\newcommand{\wt}{\widetilde}

\newcommand{\ov}{\overline}

\newcommand{\ms}{\medskip}

\mathchardef\hyp="2D

\newcommand{\qa}{{\quad \text{and} \quad}}
\newcommand{\qqa}{{~ \text{ and } ~}}

\newcommand{\mat}[4]{
 \left(  \begin{smallmatrix} #1 & #2 \\ #3 & #4 \end{smallmatrix} \right)}

\newcommand{\vect}[2]{
 \left(  \begin{smallmatrix} #1 \\ #2 \end{smallmatrix} \right)}

\newcommand{\br}[1]{\langle #1 \rangle}

\newcommand{\plims}[1]{\underset{\leftarrow ~{#1}}{\tn{lim}}\;}

\newcommand{\Qdiv}[1]{\Div^0_{\textnormal{cusp}}(X_0(#1))_{\Q}}
\newcommand{\Qdivv}[1]{\Div_{\textnormal{cusp}}(X_0(#1))_{\Q}}

\newcommand{\imply}{\Longrightarrow}

\newcommand{\bde}{{\boldsymbol{\varepsilon}}}

\begin{document}                                                            
\title{The rational torsion subgroup of $J_0(N)$}

\author{Hwajong Yoo}
\address{College of Liberal Studies and Research Institute of Mathematics, Seoul National University, Seoul 08826, South Korea}
\email{\textnormal{hwajong@snu.ac.kr}}                                                                  
\maketitle
\begin{abstract}
Let $N$ be a positive integer and let $J_0(N)$ be the Jacobian variety of the modular curve $X_0(N)$. 
For any prime $p\ge 5$ whose square does not divide $N$, we prove that 
the $p$-primary subgroup of the rational torsion subgroup of $J_0(N)$ is equal to that of the rational cuspidal divisor class group of $X_0(N)$, which is explicitly computed in \cite{Yoo9}. 
Also, we prove the same assertion holds for $p=3$ under the extra assumption that 
either $N$ is not divisible by $3$ or there is a prime divisor of $N$ congruent to $-1$ modulo $3$.
\end{abstract}
\setcounter{tocdepth}{1}
\tableofcontents

\section{Introduction}
Let $N$ be a positive integer, and let $X_0(N)$ be Shimura's canonical model over $\Q$ of the complete modular curve associated to the congruence subgroup $\G_0(N) \subseteq \SL_2(\Z)$. 
Let $J_0(N)=\Pic^0(X_0(N))$ be the Jacobian variety of the modular curve $X_0(N)$.
In this paper, we would like to understand the rational torsion subgroup of $J_0(N)$ for any positive integer $N$, which is denoted by $J_0(N)(\Q)_\tor$. 
When $N$ is a prime, Ogg conjectured the following \cite[Conj. 2]{Og75}, which was proved by Mazur \cite[Th. 1]{M77}.
\begin{theorem}[Mazur]
Let $N\geq 5$ be a prime number, and let $n=\tn{numerator}\left(\frac{N-1}{12}\right)$. The rational torsion subgroup $J_0(N)(\Q)_\tor$ is a cyclic group of order $n$, generated by the linear equivalence class of the difference of the two cusps $(0)-(\infty)$.
\end{theorem}

As a natural generalization of Ogg's conjecture, one may propose the following: For any positive integer $N$, the rational torsion subgroup $J_0(N)(\Q)_\tor$ is contained in a subgroup of $J_0(N)(\ov{\Q})$ generated by the linear equivalence classes of the differences of the cusps, which is called the \textit{cuspidal subgroup} of $J_0(N)$ and denoted by $\scC_N$. By the theorem of Manin \cite[Cor. 3.6]{Ma72} and Drinfeld \cite{Dr73}, the linear equivalence class of the difference of any two cusps is torsion, and hence
\begin{equation*}
\scC_N(\Q):= \scC_N \cap J_0(N)(\Q) \subseteq J_0(N)(\Q)_\tor,
\end{equation*}
where $\scC_N(\Q)$ is called the \textit{rational cuspidal subgroup} of $J_0(N)$.
So a generalization can be written as follows.
\begin{conjecture}[Generalized Ogg's conjecture]\label{conjecture : GOC}
For any positive integer $N$, we have
\begin{equation*}
 \scC_N(\Q)= J_0(N)(\Q)_\tor.
\end{equation*}
\end{conjecture}

In contrast to the case of prime level, it is not easy to compute the rational cuspidal subgroup $\scC_N(\Q)$ of $J_0(N)$ in general.
On the other hand, it is possible to compute a subgroup of $\scC_N(\Q)$ generated by the linear equivalence classes of the degree $0$ rational cuspidal divisors on $X_0(N)$, which is called the  \textit{rational cuspidal divisor class group} of $X_0(N)$ and denoted by $\scC(N)$.
Here, by a \textit{rational cuspidal divisor}, we mean a divisor supported only on the cusps and fixed under the action of $\Gal(\ov{\Q}/\Q)$.

Note that there may exist a cuspidal divisor $D$ which is not rational but $[D] \in \scC_N(\Q)$.
For instance, two cusps $\vect 1 3$ and $\vect 2 3$ of $X_0(9)$ are defined over $\Q(\sqrt{-3})$ and they are conjugate of each other. 
(For unfamiliar notation, see Section \ref{section: RCD}.) Since $X_0(9)$ has genus $0$, the Jacobian variety $J_0(9)$ is trivial.
Thus, a cuspidal divisor $\vect 1 3-\vect 2 3$, which is not rational, is indeed a rational point of the trivial Jacobian variety $J_0(9)$.
In this example, a cuspidal divisor $\vect 1 3-\vect 2 3$ is equivalent to a rational cuspidal divisor $(0)-(\infty)$. There are similar examples even when $J_0(N)$ is not trivial. So one may wonder such a phenomenon always occurs. 
In other words, if the equivalence class of a non-rational cuspidal divisor $D$ is fixed under the action of $\Gal(\ov{\Q}/\Q)$, then should $D$ be equivalent to some cuspidal divisor which is already fixed under the action of $\Gal(\ov{\Q}/\Q)$?

Motivated by a question of Ribet, the author proposed the following \cite[Conj. 1.3]{Yoo9}.
\begin{conjecture}\label{conjecture : yoo}
For any positive integer $N$, we have
\begin{equation*}
\scC(N) = \scC_N(\Q).
\end{equation*}
\end{conjecture}

\ms
Before proceeding, we recall several results on the conjectures above. 
For ease of notation, let
\begin{equation*}
\cA_p=\scC(N)[p^\infty] \qa \cB_p=J_0(N)(\Q)_\tor[p^\infty]
\end{equation*}
denote the $p$-primary subgroups of $\scC(N)$ and $J_0(N)(\Q)_\tor$, respectively.
\begin{enumerate}
\item\label{1}
Let $N=\ell^r$ with $r\ge 2$. 
Suppose that $\ell\ge 5$ is a prime such that $\ell \not\equiv 11 \pmod {12}$. 
Then we have $\cA_p=\cB_p$ for any odd prime $p \ne \ell$
by Lorenzini (1995) \cite[Th. 4.6]{Lo95}.

\item\label{2}
Let $N=\ell^r$ with $r\ge 2$. Suppose that $\ell$ is any odd prime.
Then for any prime $p$ not dividing $6\ell$, we have $\cA_p=\cB_p$ by Ling (1997) \cite[Th. 4]{Li97}.
Moreover, if $\ell \ge 5$ and $r=2$, then we have $\cA_3=\cB_3$ by \textit{loc. cit.}

\item\label{3}
Let $N$ be a squarefree integer. Then for any odd prime $p\neq \gcd(3, N)$, we have $\cA_p=\cB_p$ by Ohta (2014) \cite[Th.]{Oh14}.

\item\label{4}
Let $N=3\ell$ for a prime $\ell\ge 5$ such that either $\ell\not\equiv 1 \pmod 9$, or $\ell\equiv 1 \pmod 9$ and $3^{(\ell-1)/3} \not\equiv 1 \pmod \ell$. Then we have $\cA_3=\cB_3$ by the author (2016) \cite[Th. 1.3]{Yoo5}.

\item
Let $N$ be a positive integer. Then for any prime $p$ not dividing $N\prod_{\ell \mid N} (\ell^2-1)$, we have $\cA_p=\cB_p=0$ by Ren (2018) \cite[Th. 1.2]{Re18}.

\item
For some explicit values of $N$, one can compute $J_0(N)(\Q)_\tor$ by counting the numbers of $\F_p$-points on $J_0(N)_{/\F_p}$ for some primes $p$ not dividing $N$. So one may directly verify $\scC(N)=J_0(N)(\Q)_\tor$ as the group $\scC(N)$ is already computed (cf. \cite{Yoo9}). For instance, 
it is known that $\scC(N)=J_0(N)(\Q)_\tor$ for the following values of $N$:
\begin{itemize}
\item
$N=11, 14, 15, 17, 19, 20, 21, 24, 27, 32, 36, 49$ by Ligozat (1975) \cite[Th. 5.2.5]{Li75}.
\item
$N=125$ by Poulakis (1987) \cite[Prop. 3.2]{Po87}.

\item
$N=34, 38, 44, 45, 51, 52, 54, 56, 64, 81$ by Ozman and Siksek (2019) \cite[Th. 1.2]{OS19}.
\item
$N=57, 65$ by Box (2019) \cite[Lem. 3.2]{Box}.
\end{itemize}

\item
Conjecture \ref{conjecture : yoo} is fully known for some cases. More specifically, we have $\scC(N)=\scC_N(\Q)$ for the following values of $N$:
\begin{itemize}
\item
$N=n^2M$ with $n~|~24$ and $M$ squarefree by Wang--Yang (2020) \cite[Th. 3]{WY20}. 

\item
$N=p^2M$ with $p$ a prime and $M$ squarefree by Guo--Yang--Yoo--Yu (2021) \cite[Th. 1.4]{GYYY}.

\item
$N=Mp_1^{r_1}p_2^{r_2}$ with $M$ squarefree, $p_i$ odd primes and $r_i\geq 2$ by Yoo--Yu (2022) \cite{YY}.
\end{itemize}
\end{enumerate}

\ms
The main result of this paper is the following. 
\begin{theorem}\label{theorem : main theorem}
Let $N$ be a positive integer, and suppose that $p$ is any odd prime whose square does not divide $N$. 
Then for any $p\ge 5$, we have
\begin{equation*}
\scC(N)[p^\infty]=J_0(N)(\Q)_\tor[p^\infty].
\end{equation*}
Also, for a prime $p=3$ the equality holds under the extra assumption that 
either $N$ is not divisible by $3$ or there is a prime divisor of $N$ congruent to $-1$ modulo $3$. 
\end{theorem}

\ms
We make some remarks on our main theorem.
\begin{enumerate}[(a)]
\item
Let $N=\ell^r$ be a prime power with $r\ge 2$. Then for an odd prime $p\neq \ell$, we have
\begin{equation}\label{eqn:pp}
\scC(\ell^r)[p^\infty] = J_0(\ell^r)(\Q)_\tor[p^\infty].
\end{equation}
This was already known under the assumption that $\ell \not\equiv 11 \pmod {12}$ by \eqref{1}. 
Suppose that $\ell \equiv 11 \pmod {12}$ and $\ell\ge 5$. Then \eqref{eqn:pp} was known for an odd prime $\ell$ under the assumption that either $p \ge 5$ or $p=3$ and $r=2$ by \eqref{2}.
Such an extra assumption is not necessary any more. Namely, we now know that
\begin{itemize}
\item
the $3$-primary part of $\scC(\ell^r)$ is equal to that of $J_0(\ell^r)(\Q)_\tor$ for any odd prime $\ell\ge 5$ and $r\ge 3$; and
\item
the odd parts of $\scC(2^r)$ and $J_0(2^r)$ coincide. Thus, both are zero as $\scC(2^r)$ is a $2$-group by \cite[Th. 10]{RW15}.
\end{itemize}

\item
Our proof uses the theory of Eisenstein ideals, so it is completely different from those of \eqref{1} 
and \eqref{2}.

\item
Let $N=3M$ be a squarefree integer, and $p=3$. Then we have
\begin{equation*}
\scC(N)[p^\infty] = J_0(N)(\Q)_\tor[p^\infty]
\end{equation*}
under the assumption that $M$ is divisible by a prime congruent to $-1$ modulo $3$. 
This improves \eqref{3} and \eqref{4}.

\item
Let $N$ be any positive integer. Since the structure of $\scC(N)[p^\infty]$ is completely determined for any prime $p$ (including $p=2$) in \cite{Yoo9}, we now fully understand the structure of the $p$-primary subgroup of the rational torsion subgroup of $J_0(N)$ under the assumptions above.

\item
To the best of this author's knowledge, there is no (partial) results about Conjecture \ref{conjecture : GOC} which do not prove Conjecture \ref{conjecture : yoo} simultaneously. This was also one of motivations for Conjecture \ref{conjecture : yoo}.
\end{enumerate}

\ms
The paper will proceed as follows. In Section \ref{section: strategy}, we give a version of Ohta's proof that
\begin{equation}\label{aa}
\scC(N)[p^\infty]=J_0(N)(\Q)_\tor[p^\infty]
\end{equation}
when $N$ is squarefree and $p$ is an odd prime such that $p \neq \gcd(3, N)$. Also, we prove \eqref{aa} for $p=3$ when $N=3M$ is squarefree and $M$ is divisible by a prime congruent to $-1$ modulo $3$. 
In Section \ref{section: RCD}, we construct a certain rational cuspidal divisor $C_0^{\bde}$ annihilated by an Eisenstein ideal $\fI_0^{\bde}$,
and compute the order of its linear equivalence class in $J_0(N)$. 
In Section \ref{section: index}, we construct an Eisenstein series $E_0^{\bde}$ and compute its residue at various cusps of $X_0(N)$. Using this, we compute the $p$-part of the index of an Eisenstein ideal $\fI_0^{\bde}$ when $p$ is an odd prime whose square does not divide $N$.
In Section \ref{section: multi}, we study the kernel of an Eisenstein ideal in characteristic $p$.
Lastly, combining all the results we prove Theorem \ref{theorem : main theorem} in Section \ref{section: proof}.

\ms
\subsection{Notation}\label{section: notation}
Throughout the whole paper, let $p$, $\ell$ and $\ell_i$ denote primes. 
For a positive integer $N$, we write
\begin{equation*}
N=N_1 N_2
\end{equation*}
with $\gcd(N_1, N_2)=1$. Also, we write
\begin{equation*}
N_1=\prod_{i=1}^t \ell_i \qa N_2=\prod_{j=t+1}^u \ell_j^{r_j}
\end{equation*}
with integers $r_j \ge 2$ for all $t<j \le u$.
So in our convention, $u$ is the number of the prime divisors of $N$ and $t$ is the number of the prime divisors of $N$ which exactly divide $N$. Moreover, $N$ is squarefree if and only if $t=u$, i.e., $N_2=1$.
As noticed by many mathematicians (for example, see \cite[pg. 275]{Oh14} and \cite[Sec. 1.4.3]{WWE21}), when $\ell$ exactly divides $N$, the Atkin--Lehner operator $w_\ell$ is more convenient for studying Eisenstein ideals than the Hecke operator $T_\ell$. On the other hand, we believe that the Hecke operator $T_\ell$ is still important when $\ell^2$ divides $N$, so we define the Hecke algebra $\T(N)$ of level $N$ as follows:
\begin{equation*}
\T(N):=\Z[T_\ell, w_{\ell_i} : \ell \nmid N_1 \qqa 1\leq i \leq t] \subset \End(J_0(N)).
\end{equation*}
We simply write $\T$ for $\T(N)$ if there is no confusion.

For the sake of the readers, we discuss our conventions on the degeneracy maps, the Hecke operators and the Atkin--Lehner operators. 
We define the maps $\alpha_\ell(N)$ and $\beta_\ell(N)$ from $\G_0(N\ell) \backslash \cH$ to $\G_0(N) \backslash \cH$ by
\begin{equation*}
\alpha_\ell(N)(\tau \mod {\G_0(N\ell)}):=\tau \mod {\G_0(N)}, \quad
\beta_\ell(N)(\tau \mod {\G_0(N\ell)}):=\ell \tau \mod {\G_0(N)},
\end{equation*}
where $\tau$ is an element of the complex upper half plane $\cH$. They induce the degeneracy maps from $X_0(N\ell)$ to $X_0(N)$, which we use the same notation. Note that they are both defined over $\Q$. The ``usual'' modular interpretations are as follows: For a pair $(E, C)$ of an elliptic curve $E$ and a cyclic subgroup $C$ of $E$ of order $N\ell$, we have
\begin{equation*}
\alpha_\ell(N)(E, C)=(E, C[N]) \qa \beta_\ell(N)(E, C)=(E/{C[\ell]}, C/{C[\ell]}).
\end{equation*}
These degeneracy maps induce the maps between Jacobian varieties:
\begin{equation*}
\a_\ell(N)_*, \b_\ell(N)_* : J_0(N\ell) \rightrightarrows J_0(N) \qa \a_\ell(N)^*, \b_\ell(N)^* : J_0(N) \rightrightarrows J_0(N\ell).
\end{equation*}
(For more detail, see \cite[Sec. 13]{MR91} or \cite[Sec. 2.2]{Yoo9}.) We then take the definition of the Hecke operator $T_\ell$ as 
\begin{equation}\label{eqn: Hecke}
T_\ell :=\beta_\ell(N)_* \circ \alpha_\ell(N)^* : J_0(N) \to J_0(N).
\end{equation}
This definition also induces the Hecke operator $T_\ell$ acting on the space of modular forms for $\G_0(N)$ (cf. \cite[pg. 217]{MR91}).

Suppose that $\ell=\ell_i$ for some $1\le i \le t$. Namely, $N=M\ell$ with $\gcd(M, \ell)=1$. Then there is the Atkin--Lehner involution $\tn{w}_\ell$ on $X_0(N)$, which is defined by the following modular interpretation: For a pair $(E, C)$ of an elliptic curve $E$ and a cyclic subgroup $C$ of order $N$, we have $\tn{w}_\ell(E, C)=(E/{C[\ell]}, C/{C[\ell]} \oplus E[\ell]/{C[\ell]})$. They satisfy the relations: 
\begin{equation*}
\b_\ell(M)=\a_\ell(M) \circ \tn{w}_\ell \qa \a_\ell(M)=\b_\ell(M) \circ \tn{w}_\ell.
\end{equation*}
We define the Atkin--Lehner operator $w_\ell$ acting on $J_0(N)$ as the map induced from $\tn{w}_\ell$ by the Picard functoriality.
Namely, we have 
\begin{equation}\label{eqn: AL}
w_\ell \circ \a_\ell(M)^* = \b_\ell(M)^* \qa w_\ell \circ \b_\ell(M)^* = \a_\ell(M)^*.
\end{equation}

\begin{remark}
In some literature, $T_\ell$ is defined as $\alpha_\ell(M)_* \circ \beta_\ell(M)^*$, which is the transpose of our definition. 
If we replace our definition, all the arguments in the paper can be changed appropriately without further difficulty (cf. Remarks \ref{rem: new Hecke RCD} and \ref{rem: new Hecke Eisenstein}).
\end{remark}

\ms
\section{The base case: Squarefree level}\label{section: strategy}
In this section, we prove Theorem \ref{theorem : main theorem} when $N$ is squarefree. Namely, we assume that $N$ is squarefree and prove that 
\begin{equation}\label{eqn}
\scC(N)[p^\infty]=J_0(N)(\Q)_\tor[p^\infty] 
\end{equation}
under the assumption that either one of the following holds:
\begin{enumerate}
\item
$p$ is an odd prime different from $\gcd(3, N)$.
\item
$p=3$ and $N=3M$, where $M$ is divisible by a prime congruent to $-1$ modulo $3$. 
\end{enumerate}

Note that since we assume that $N$ is squarefree, Conjecture \ref{conjecture : yoo} is obviously true.
Note also that Ohta proved \eqref{eqn} under the first assumption. Here, we just rearrange Ohta's arguments in order to explain our natural generalization in Sections \ref{section: index} and \ref{section: multi}. Finally, we follow Ohta's suggestion and slightly extend his result.

\ms
Before proceeding, we fix some notations: From now on, $p$ always denotes an odd prime. Let 
\begin{equation*}
N=N_1=\prod_{i=1}^t \ell_i
\end{equation*}
be a squarefree integer for some $t\geq 1$. Let
\begin{equation*}
\{\pm 1\}^t :=\{(\ve_1, \cdots, \ve_t) : \ve_i =1 \text{ or } -1 \text{ for all } 1\le i \le t\}.
\end{equation*}
Let
\begin{equation*}
\T_p:=\T \otimes_\Z \Z_{(p)}
\end{equation*}
be the localization of $\T$. (Here, $\Z_{(p)}$ is the localization of $\Z$ at the prime ideal $(p)$.) Also, let 
\begin{equation*}
\fI^N:=(T_\ell-\ell-1 : \ell \nmid N)
\end{equation*}
be a (minimal) Eisenstein ideal of $\T_p$. Furthermore, for any $\bde:=(\varepsilon_1, \dots, \varepsilon_t) \in \{\pm 1\}^t$, let
\begin{equation*}
\fI^{\bde}:=(\fI^N, w_{\ell_i}-\varepsilon_i : 1\leq i \leq t) \qa \fm^\bde:=(p, \fI^\bde)
\end{equation*}
be Eisenstein ideals of $\T_p$. Finally, let
\begin{equation*}
\cA_p=\scC(N)[p^\infty] \qa  \cB_p=J_0(N)(\Q)_\tor[p^\infty].
\end{equation*}

\ms
Now, we outline a proof of \eqref{eqn}. By the theorem of Manin \cite[Cor. 3.6]{Ma72} and Drinfeld \cite{Dr73}, we easily have 
\begin{equation*}
\cA_p \subseteq \cB_p
\end{equation*}
and so it suffices to show that $\#\cA_p \geq \# \cB_p$. Before proceeding, we remark that following Mazur we consider $\cB_p$ and its subset $\cA_p$ as modules over $\T_p$. 

Firstly, $\cB_p$ is annihilated by $\fI^N$ by the Eichler--Shimura relation (Lemma \ref{lemma: ES relation}). Thus, we can regard $\cB_p$ and its submodule $\cA_p$ as modules over $\T_p/{\fI^N}$. 
Secondly, as we assume that $p$ is odd, the quotient ring $\T_p /{\fI^N}$ decomposes into the product of $\T_p/{\fI^\bde}$ (Lemma \ref{lemma: decomposition}).
Hence we can also decompose $\cA_p$ and $\cB_p$ into the direct sums of $\cA_p[\fI^\bde]$ and $\cB_p[\fI^\bde]$, respectively. 
Thus, it suffices to prove that $\# \cA_p[\fI^\bde] \ge \# \cB_p[\fI^\bde]$ for all $\bde \in \{\pm 1\}^t$.
If 
\begin{equation*}
\bde=\bde_+:=(+1, \dots, +1),
\end{equation*}
then we have $\T_p=\fI^{\bde}$ and hence $\cA_p[\fI^{\bde}]=\cB_p[\fI^{\bde}]=0$ (Lemma \ref{lemma: do not occur}). Therefore we can assume that $\bde\ne \bde_+$.
Thirdly, we construct a (rational) cuspidal divisor $C^\bde$ (Definition \ref{def: RCD}). 
By its construction, we can prove that its linear equivalence class $[C^\bde]$ is annihilated by $\fI^\bde$ (Lemma \ref{lemma: C annihilated by I}). The order of $[C^\bde]$ in $J_0(N)(\Q)_\tor$ is equal to 
\begin{equation*}
\fn^{\bde}:=\tn{numerator}\left(\frac{1}{24}\prod_{i=1}^t (\ell_i+\varepsilon_i)\right)
\end{equation*}
or $2\fn^\bde$ (Lemma \ref{lemma: order}).
Thus, we have $\# \cA_p[\fI^\bde] \ge p^{\val_p(\fn^\bde)}$.
Fourthly, we study Eisenstein series annihilated by $\fI^\bde$ and prove that 
\begin{equation*}
\T_p/{\fI^{\bde}} \simeq \Z_{(p)}/{\fn^\bde\Z_{(p)}}
\end{equation*}
(Theorem \ref{thm: Ohta index}).
Finally, we prove that $\cB_p[\fI^\bde]$ is cyclic as a $\T_p/{\fI^\bde}$-module under our assumptions (Theorem \ref{thm: Ohta multi}), and so $\# \cB_p[\fI^\bde] \le \# (\T_p/{\fI^\bde})=p^{\val_p(\fn^\bde)}$. This completes the proof. \qed

\begin{remark}\label{remark: Ohta} 
As noticed by Ohta \cite[pg. 317]{Oh14}, the first assumption that $p \ne \gcd(3, N)$ is only necessary in the last step. So, if we prove that $\cB_p[\fI^\bde]$ is cyclic without it, then \eqref{eqn} holds for all odd primes $p$ whenever $N$ is squarefree. 
Indeed, our new contribution in this section is proving that $\cB_3[\fI^\bde]$ is cyclic for all $\bde \in \{\pm 1\}^t$ under the second assumption, which implies the result.
\end{remark}

\begin{remark}
Ohta didn't consider the divisor $C^{\bde}$. Instead, he use Takagi's result on the cuspidal class number of $X_0(N)$ \cite{Ta97}, which says that
\begin{equation*}
\# \cA _p = \prod_{\bde} p^{\val_p(\fn^\bde)}.
\end{equation*}
This is enough to conclude that $\cA_p[\fI^\bde]=\cB_p[\fI^\bde]$ and it is a free module of rank $1$ over $\T_p/{\fI^\bde}$. Hence there is a cuspidal divisor $D$ such that $[D]$ is a generator of $\cA_p[\fI^\bde]$. Finding such a divisor $D$ (which is our $C^\bde$) was one of the motivations of this paper.
\end{remark}

In the rest of the section, we provide proofs (or explicit references) of the results used in the outline. All of them (except the last case of Theorem \ref{thm: Ohta multi}) are either well-known or appeared in \cite{M77, Oh14}; they are included only for the sake of completeness. 
\begin{lemma}\label{lemma: ES relation}
Let $N$ be a positive integer. Then $\cB_p$ is annihilated by $T_\ell-\ell-1$ for any primes $\ell$ not dividing $N$.
\end{lemma}

\begin{proof}
Let $\ell$ be a prime not dividing $N$. Since $J_0(N)$ has good reduction at $\ell$, 
the special fiber of the N\'eron model of $J_0(N)$ at $\ell$ is $J_0(N)_{/\F_\ell}$. Thus, there is a Hecke-equivariant specialization map:
\begin{equation*}
\iota_\ell : J_0(N)(\Q)_\tor\to J_0(N)_{/\F_\ell}(\F_\ell).
\end{equation*}
Suppose that $\ell=2$. Then we have $\ell \ne p$ since we assume that $p$ is odd. Thus, the restriction of $\iota_\ell$ on $\cB_p$ is injective by \cite[Th. 1]{ST68}. If $\ell$ is odd, then $\iota_\ell$ is already injective by Katz \cite[App.]{Ka81}. 
Consequently, the restriction of $\iota_\ell$ on $\cB_p$ is injective under our assumption.
Since $\iota_\ell$ is Hecke-equivariant, it suffices to show that the $\F_\ell$-points on $J_0(N)_{/\F_\ell}$ are annihilated by $T_\ell-\ell-1$.
By the Eichler--Shimura congruence relation (cf. \cite[Sec. 8.7]{DS05}), the Hecke operator $T_\ell$ acts on $J_0(N)_{/\F_\ell}$ as the sum of the Frobenius morphism and its transpose. Since the Frobenius morphism acts trivially on $J_0(N)_{/\F_\ell}(\F_\ell)$, the assertion follows.
\end{proof}

The following is a variant of (2.4.4) on \cite[pg. 300]{Oh14}.
\begin{lemma}\label{lemma: decomposition}
We have
\begin{equation*}
\T_p/{\fI^N} \simeq  \prod_{\bde \in \{\pm 1\}^t} \T_p/{\fI^{\bde}}.
\end{equation*}
\end{lemma}
\begin{proof}
Let $\ell=\ell_i$ for some $1\leq i \leq t$ and let $\fm$ be a maximal ideal of $\T_p$ containing $\fI^N$. Also, let
\begin{equation*}
\T_{\fm}:= \plims k \T_p/{\fm^k} 
\end{equation*}
be the completion of $\T_p$ at $\fm$. 
Since $w_\ell^2-1$ is zero as an element of $\End(J_0(N))$, we have $w_\ell-\varepsilon \in \fm$ for some $\varepsilon \in \{\pm 1\}$. Thus, we have 
\begin{equation*}
\fm=\fm^{\bde} \quad \text{ for some }~ \bde=(\ve_1, \ve_2, \dots, \ve_t) \in \{\pm 1\}^t.
\end{equation*}

Since $p$ is odd, $w_{\ell_i}+\ve_i \not\in \fm$ and so it is a unit in $\T_\fm$. Hence we have $w_{\ell_i}-\ve_i=0 \in \T_\fm$ for all $1\leq i\leq t$. Thus, we have $\T_{\fm}/{\fI^{\bde}}=\T_{\fm}/{\fI^N}$.
Also, since $\fm^k \subset \fI^{\bde}$ for sufficiently large $k$ (or just taking $k$ as the $p$-adic valuation of the index of $\fI^{\bde}$), we have
\begin{equation*}
\T_{\fm}/{\fI^{\bde}}=\underset{\leftarrow ~k}{\tn{lim}} ~~\T_p/{(\fm^k, \fI^{\bde})}  \simeq \T_p/{\fI^{\bde}}.
\end{equation*}
Thus, we have
\begin{equation*}
\T_p/{\fI^N} \simeq \prod_{\fI^N \subset \fm \subset \T_p \tn{ maximal}} \T_{\fm}/{\fI^N}= \prod_{\bde \in \{\pm 1\}^t} \T_{\fm^{\bde}}/{\fI^{\bde}} \simeq \prod_{\bde \in \{\pm 1\}^t} \T_p/{\fI^{\bde}}.
\end{equation*}
This completes the proof.
\end{proof}

Note that the operators $T_\ell$ ($\ell \ne \ell_i$) and $w_i$ are all congruent to integers modulo $\fI^\bde$, there is a surjection $\Z_{(p)} \to \T_p/{\fI^\bde}$. By the Ramanujan--Petersson bound, we further have
\begin{equation*}
\T_p/{\fI^{\bde}} \simeq \Z_{(p)}/{n\Z_{(p)}}
\end{equation*}
for some integer $n\ge 1$, which is called the ($p$-part of the) \textit{index of an Eisenstein ideal} $\fI^\bde$.
One of the key results of \cite{Oh14} is the computation of this index.

\begin{lemma}\label{lemma: do not occur}
If $N$ is squarefree and $\bde=\bde_+$, then we have $\T_p = \fI^\bde$.
\end{lemma}
\begin{proof}
If $N$ is a prime, then Mazur proved that $\fI^N$ contains $T_N-1=w_N+1$ (cf. \cite[Prof. 3.19]{CE05}). Thus, the assertion follows. In general, it easily follows from \cite[Th. 3.1.3]{Oh14}.
\end{proof}

\begin{theorem}\label{thm: Ohta index}
We have $\T_p/{\fI^{\bde}} \simeq \Z_{(p)}/{\fn^\bde\Z_{(p)}}$.
\end{theorem}
\begin{proof}
This easily follows from \cite[Th. 3.1.3]{Oh14} as $\fn^\bde$ is equal to $c(N; \bde)$ in \textit{loc. cit.} up to powers of $2$.
\end{proof}

\begin{theorem}\label{thm: Ohta multi}
Suppose that either one of the following holds:
\begin{enumerate}
\item
$p$ does not divide $N$.
\item
$p\ge 5$ divides $N$.
\item
$p=3$ divides $N$, and $N$ is divisible by a prime congruent to $-1$ modulo $3$.
\end{enumerate}
Then $\cB_p[\fI^\bde]$ is cyclic as a $\T_p/{\fI^\bde}$-module.
\end{theorem}
\begin{proof}
If $\fm^\bde$ is not maximal, in which case there does not exist cusp forms congruent to Eisenstein series (annihilated by $\fI^\bde$) modulo $p$, 
then we have $\T_p=\fI^\bde$ and hence $\cB_p[\fI^\bde]=0$. So we assume that $\fm^\bde$ is maximal. Since $\fm^\bde$ is the only maximal ideal containing $\fI^\bde$, by Nakayama's lemma it suffices to show that the dimension of $\cB_p[\fm^\bde]$ over $\T_p/{\fm^\bde} \simeq \F_p$ is at most $1$. Since $p$ is odd, the specialization map
\begin{equation}\label{eqn33}
\iota_p: J_0(N)(\Q)_\tor \to \cJ_p(\F_p)
\end{equation}
is injective by Katz \cite[App.]{Ka81}, where $\cJ_p$ is the special fiber of the N\'eron model of $J_0(N)$ at $p$. 
Hence it suffices to show that $\cJ_p(\F_p)[\fm^\bde]$ is at most of dimension $1$. 
\begin{enumerate}
\item
Suppose that $p$ does not divide $N$. Then it follows by Proposition 3.5.4 of \cite{Oh14}. 
The key ingredients of the proof are the Cartier morphism in characteristic $p$ and mod $p$ multiplicity one result for differentials.

\item
Suppose that $p\ge 5$ divides $N$. Then it follows by Proposition 3.5.9 of \textit{op. cit.}
A generalization of this proof will be explained in Section 5 below, so we just point out the key step where the assumption $p\ge 5$ is used.
Let $\Phi_p$ be the component group of the spcial fiber $\cJ_p$. Then it is well-known that $\Phi_p$ can be decomposed into the direct sum of a large cyclic group and some extra elementary $2$-groups and $3$-groups (cf. \cite{Ed91} or \cite{KY18}).
If $p\ge 5$, then these extra elementary groups do not contribute the kernel of $\fm^\bde$ on $\cJ_p$ and the assertion easily follows. 

\item
Suppose that $p=3$ divides $N$. Then the arguments in (2) work verbatim as long as $\Phi_3$ does not have these extra elementary $3$-groups. It exactly occurs when all the prime divisors of $N/3$ are congruent to $1$ modulo $3$ (cf. \cite[pg. 343]{KY18}). Thus, if $N/3$ is divisible by a prime congruent to $-1$ modulo $3$, then $\Phi_3$ does not have these extra elementary $3$-groups and the argument of Ohta in Lemma 3.5.8 of \cite{Oh14} works \textit{mutatis mutandis}. 
\end{enumerate}
This completes the proof.
\end{proof}

\begin{remark}
Even though $\Phi_3$ has extra elementary $3$-groups, if one could prove that
the image of 
\begin{equation*}
\iota_3: J_0(N)(\Q)_\tor[\fm^\bde] \to \cJ_3(\F_3)
\end{equation*}
has the trivial intersection with these extra elementary $3$-groups of $\Phi_3$ (for all $\bde \in \{\pm 1\}^t$), the same argument of Ohta would work and we would prove \eqref{eqn} for $p=3$ without our redundant assumptions.
\end{remark}

\ms
\section{Rational cuspidal divisors}\label{section: RCD}
In this section, we construct a rational cuspidal divisor $C_0^{\bde}$ on $X_0(N)$ for any $\bde \in \{ \pm 1\}^t$.
Throughout the section, there is no restriction on $N$.

\ms
To begin with, we recall some results about the cusps of $X_0(N)$. Let $N$ be any positive integer. Recall that $X_0(N)$ is Shimura's canonical model over $\Q$ of the complete modular curve $\G_0(N) \backslash \cH \cup \bP^1(\Q)$ for the congruence subgroup $\G_0(N)$. As representatives of the cusps of $X_0(N)$, we use Ogg's notation. Namely, we denote a cusp of $X_0(N)$ by a vector $\vect x d$, where $d$ is a positive divisor of $N$ and $1\le x \le d$ is an integer relatively prime to $d$.
(Note that a cusp of $X_0(N)$ is an element of $\G_0(N) \backslash \bP^1(\Q)$. Here, a vector $\vect x d$ corresponds to the equivalence class of a rational number $\frac{x}{d}$.)
Such two cusps $\vect x d$ and $\vect y e$ are equivalent if and only if $d=e$ and $x\equiv y \pmod z$, where 
\begin{equation*}
z:=\gcd(d, N/d).
\end{equation*}
Thus, if we write a cusp of $X_0(N)$ as the form above, then the divisor $d$ of $N$ is uniquely determined. So we say that a cusp of $X_0(N)$ is \textit{of level $d$} if it is written as $\vect x d$.
For instance, the cusps $0$ and $\infty$ are equivalent to $\vect 1 1$ and  $\vect 1 N$, respectively. Thus, the cusp $0$ is of level $1$ and the cusp $\infty$ is of level $N$. (Warning: When we say that a cusp is of level $d$, it should be understood that the modular curve is given, and $d$ is always a positive divisor of the level of a given modular curve.)
Any cusp of level $d$ is defined over $\Q(\mu_z)$ and the action of $\Gal(\Q(\mu_z)/\Q)$ on the set of all cusps of level $d$ is simply transitive \cite[Th. 2.17]{Yoo9}. For instance, the cusps $0$ and $\infty$ are defined over $\Q$. Thus, the divisor
\begin{equation*}
(P_d):=\sum_{c \in \{\tn{cusps of level }d\}} c=\sum_{\substack{1\le x \le d, ~\gcd(x, d)=1,\\ \text{$x$ taken modulo $z$}}} \vect x d
\end{equation*}
is defined over $\Q$. Since the degree of $(P_d)$ is $\p(z)$, where $\p$ is Euler's totient function,
the divisor
\begin{equation*}
C_d:=\p(z) \cdot (P_1) -(P_d)
\end{equation*}
is of degree $0$. Let $\Qdivv N$ be the group of the rational cuspidal divisors on $X_0(N)$.
Then we have
\begin{equation*}
\Qdivv N = \left\{\sum_{1\le d \mid N} a(d) \cdot (P_d) : a(d) \in \Z \right\}.
\end{equation*}
For more details, see \cite[Sec. 2]{Yoo9}.

\ms
Let $\cS_2(N)_\Q$ be the $\Q$-vector space of dimension $\sigma_0(N)$ indexed by the divisors of $N$, and let $\cS_2(N)$ be its canonical integral lattice. In other words,
\begin{equation*}
\cS_2(N):=\left\{\sum_{1\le d\mid N} a(d) \cdot {\bf e}(N)_d : a(d) \in \Z \right\},
\end{equation*}
where ${\bf e}(N)_d$ is the unit vector in $\cS_2(N)$ whose $d$th entry is $1$ and all other entries are zero. Then there is a tautological isomorphism:
\begin{equation*}
\Phi_N : \Qdivv N \to \cS_2(N)
\end{equation*}
sending $(P_d)$ to ${\bf e}(N)_d$. By the Chinese Remainder theorem, we have a canonical isomorphism
\begin{equation*}
\cS_2(N)_\Q \simeq \motimes_{\ell \mid N} \cS_2(\ell^{\tn{val}_\ell(N)})_\Q,
\end{equation*}
and we identify both sides by insisting ${\bf e}(N)_d=\motimes_{\ell \mid N} {\bf e}(\ell^{\tn{val}_\ell(N)})_{\ell^{\tn{val}_\ell(d)}}$. 

\ms
Now, we are ready to define a rational cuspidal divisor $C_0^{\bde}$ on $X_0(N)$.
As introduced in Section \ref{section: notation}, let $N=N_1N_2$ with $\gcd(N_1, N_2)=1$. Also, we write
\begin{equation*}
N_1=\prod_{i=1}^t \ell_i \qa N_2=\prod_{j=t+1}^u \ell_j^{r_j}
\end{equation*}
with integers $r_j \ge 2$ for all $t<j \le u$.
From now on, let
\begin{equation*}
\bde=(\varepsilon_1, \dots, \varepsilon_t) \in \{\pm 1\}^t.
\end{equation*}

\begin{definition}\label{def: RCD}
For each $1\leq i \leq t$, let
\begin{equation*}
{\bf w}_i:={\bf e}(\ell_i)_1+\varepsilon_i\cdot {\bf e}(\ell_i)_{\ell_i} \in \cS_2(\ell_i).
\end{equation*}
Also, for each $t<j \leq u$, let 
\begin{equation*}
{\bf w}_j:=(\ell_j-1) \cdot {\bf e}(\ell_j^{r_j})_1-{\bf e}(\ell_j^{r_j})_{\ell_j} \in \cS_2(\ell_j^{r_j}).
\end{equation*}
Finally, let
\begin{equation*}
\C_0^{\bde}:=\motimes_{i=1}^u {\bf w}_i \in \motimes_{i=1}^u \cS_2(\ell_i^{r_i}) = \cS_2(N) \qa C_0^{\bde}:=\Phi_N^{-1}(\C_0^{\bde}) \in \Qdivv N.
\end{equation*}
In other words, if we write $C_0^{\bde}=\sum_{1\le d \mid N} a(d) \cdot (P_d)$, then we have
\begin{equation*}
a(d)=\begin{cases}
\prod_{i=1}^t \varepsilon_i^{f_i} \times \prod_{j=t+1}^u (\ell_j-1)(1-\ell_j)^{-f_j} & \text{ if $d=\prod_{i=1}^u \ell_i^{f_i}$ is squarefree},\\
\qquad 0 & \quad\text{ otherwise}.
\end{cases}
\end{equation*}
If $N$ is squarefree, we often denote $C_0^\bde$ by $C^\bde$.
\end{definition}
Note that the degree of $C_0^{\bde}$ is $0$ unless $N$ is squarefree and $\bde=\bde_+$, in which case it is $2^t$.
So, we exclude this case.  Let $[C_0^{\bde}]$ denote the linear equivalence class of $C_0^{\bde}$ in $J_0(N)$. By \cite[Lem. 2.23]{Yoo9}, we have the following.
\begin{lemma}\label{lemma: C annihilated by I}
For any $1\leq i \leq t < j \le u$ and any prime $\ell$ not dividing $N$, we have 
\begin{equation*}
w_{\ell_i}(C_0^{\bde})=\varepsilon_i \cdot C_0^{\bde}, \quad T_{\ell_j}(C_0^{\bde})=0  \qa T_\ell(C_0^{\bde})=(\ell+1) \cdot C_0^{\bde}.
\end{equation*}
Hence the same equalities hold if we replace $C_0^\bde$ by $[C_0^\bde]$.
\end{lemma}

Also, we have the following.
\begin{lemma}\label{lemma: order}
Assume $\bde\neq \bde_+$ if $N$ is squarefree. Then the order of $[C_0^{\bde}]$ in $J_0(N)$ is 
\begin{equation*}
\tn{numerator}\left(\frac{h}{24} \prod_{i=1}^t (\ell_i+\varepsilon_i) \times \prod_{j=t+1}^u \ell_j^{r_j-2}(\ell_j^2-1) \right),
\end{equation*}
where $h=2$ if $N$ is either a prime or a power of $2$, and $h=1$ otherwise.
\end{lemma}
\begin{proof}
This is an easy exercise applying Theorem 3.13 of \cite{Yoo9}. 
For the sake of the readers, we provide a complete proof.
We use the same notation as in \textit{loc. cit.}
The order of $[C_0^{\bde}]$ is equal to
\begin{equation*}
\tn{numerator}\left( \frac{\kappa(N) \times \fh(C_0^{\bde})}{24 \times \textsf{Gcd}(C_0^{\bde})} \right).
\end{equation*}

Recall that $\textsf{Gcd}(C_0^{\bde})$ is the greatest common divisor of the entries of $\Upsilon(N) \times \C_0^{\bde}$. Here, we consider $\C_0^{\bde}$ as a column vector and the matrix $\Upsilon(N)$ is defined as the tensor product of 
the matrices $\Upsilon(\ell_i)$ and $\Upsilon(\ell_j^{r_j})$ for all $1\le i \le t < j \le u$. Moreover, the matrix $\Upsilon(\ell^r)$ for a prime $\ell$ and an integer $r\ge 1$, which is indexed by the positive divisors of $\ell^r$, is defined by
\begin{equation*}
\Upsilon(\ell^r)_{\ell^i \ell^j}:=\begin{cases}
\ell & \text{ if } ~~ i=j=0 ~~\text{ or }~~ r, \\
\ell^{m(j)-1}(\ell^2+1) & \text{ if }~~1\leq i=j \leq r-1,\\
-\ell^{m(j)} & \text{ if }~~ |i-j|=1,\\
0 & \text{ if }~~ |i-j|\geq 2,
\end{cases}
\end{equation*}
where $m(f):=\min(f, \, r-f)$. 
Since $\C_0^\bde$ is also defined by the tensor product, by Theorem 3.15 of \textit{op. cit.} it suffices to compute the following: If $\ell=\ell_i$ for some $1\le i \le t$, then 
\begin{equation*}
\Upsilon(\ell) \times \vect 1 {\varepsilon_i}=\mat \ell {-1} {-1} \ell  \times \vect 1 {\varepsilon_i}=(\ell-\varepsilon_i) \vect 1 {\varepsilon_i}.
\end{equation*}
Also, if $\ell=\ell_j$ for some $t< j \le u$, then 
\begin{equation*}
\Upsilon(\ell^{r_j}) \times \left(\begin{smallmatrix} 
\ell-1 \\
-1\\
0\\
\bbO
\end{smallmatrix}\right)=\left(\begin{smallmatrix} 
\ell & -\ell & \dots \\
-1 & \ell^2+1 & \dots\\
0 & -\ell & \dots \\
\bbO & \bbO & \dots
\end{smallmatrix}\right) \times \left(\begin{smallmatrix} 
\ell-1 \\
-1\\
0 \\
\bbO
\end{smallmatrix}\right)=\ell\left(\begin{smallmatrix} 
\ell \\
-\ell-1\\
1\\
\bbO
\end{smallmatrix}\right),
\end{equation*}
where $\bbO$ denotes the zero vector of length $r_j-2$. Thus, we have 
\begin{equation*}
\frac{\kappa(N)}{\textsf{Gcd}(C_0^{\bde})}=\frac{\prod_{i=1}^t (\ell_i^2-1) \prod_{j=t+1}^u \ell_j^{r_j-1}(\ell_j^2-1)}{\prod_{i=1}^t (\ell_i-\ve_i) \prod_{j=t+1}^u \ell_j}=
\prod_{i=1}^t (\ell_i+\varepsilon_i) \times \prod_{j=t+1}^u \ell_j^{r_j-2}(\ell_j^2-1).
\end{equation*}
Since the sum of the entries of $(1, \varepsilon_i)$ (or of $(\ell_j, -\ell_j-1, 1, 0, \dots, 0)$) is even, by its definition we easily have $\fh(C_0^\bde)=1$ whenever $u\ge 2$. Suppose that $u=1$. Then by its definition, we have
\begin{equation*}
\textsf{Pw}_{\ell_1}(C_0^\bde)=\begin{cases}
-\ell_1-1 & \text{ if }~~ t=0,\\
\phantom{-a}\ve_1 & \text{ if } ~~ t=1.
\end{cases}
\end{equation*}
Thus, it is even if and only if $\ell_1$ is odd and $t=0$. So by definition, $\fh(C_0^\bde)=2$ if and only if either $N$ is a prime or $N$ is a power of $2$. Otherwise, $\fh(C_0^\bde)=1$. This completes the proof.
\end{proof}

Let $\Qdiv N$ be the group of degree $0$ rational cuspidal divisors on $X_0(N)$. 
\begin{lemma}\label{lemma: image of beta}
Let $N=M\ell$, and let $\beta:=\beta_{\ell}(M)_* : \Qdiv N \to \Qdiv M$ be the map induced by the degeneracy map $\beta_{\ell}(M) : X_0(N) \to X_0(M)$ explained in Section \ref{section: notation}.
Then we have
\begin{equation*}
\Qdiv M /{\beta(\Qdiv N)} \simeq (\zmod \ell)^k
\end{equation*}
for some integer $k\geq 0$. Moreover, if $\ell^4$ does not divide $N$, then $k=0$ and thus we have
\begin{equation*}
\beta_\ell(M)_*(\scC(N))=\scC(M).
\end{equation*}
\end{lemma}
\begin{proof}
During the proof, we identify $\Qdivv N$ and $\Qdivv M$ with $\cS_2(N)$ and $\cS_2(M)$ (by the maps $\Phi_N$ and $\Phi_M$), respectively. And we regard $\beta$ as a map from $\cS_2(N)$ to $\cS_2(M)$. 
Let $\cS_2(N)^0$ (resp. $\cS_2(M)^0$) denote the image of $\Qdiv N$ (resp. $\Qdiv M$) by $\Phi_N$ (resp. $\Phi_M$). Since $\beta(\cS_2(N)^0)=\beta(\cS_2(N)) \cap \cS_2(M)^0$,
it suffices to show that
\begin{equation*}
\cS_2(M) /{\beta(\cS_2(N))} \simeq (\zmod \ell)^k
\end{equation*}
for some integer $k\geq 0$, and $k=0$ if $\ell^4$ does not divide $N$. 

Let $M=L\ell^r$ with $\gcd(L, \ell)=1$, and let $d$ be a divisor of $L$. If $\beta(\cS_2(N)) \subset \ell \cdot \cS_2(M)$, then 
\begin{equation*}
\cS_2(M)/{\ell \cdot \cS_2(M)} \simeq (\zmod \ell)^{\sigma_0(M)} \surj \cS_2(M)/{\beta(\cS_2(N))}.
\end{equation*}
Since ${\bf e}(M)_{d\ell^f}$ are the generators of $\cS_2(M)$ (for all divisors $d$ of $L$), to prove the first assertion it suffices to show that either ${\bf e}(M)_{d\ell^f}$ or $\ell \cdot {\bf e}(M)_{d\ell^f}$ is in the image of $\beta$ for every $0\le f \le r$. 

If $r=0$, i.e., $M=L$, then we easily have $\beta({\bf e}(N)_{d\ell})=\beta({\bf e}(N)_{d})={\bf e}(M)_d$. 
Suppose that $r\ge 1$. Then by \cite[Lem. 2.21]{Yoo9}, we have
\begin{equation*}
\beta({\bf e}(N)_{d\ell^{f+1}})=\begin{cases}
\ell \cdot {\bf e}(M)_{d\ell^f} & \text{ if }~~0< f <r/2,\\
{\bf e}(M)_{d\ell^f} & \text{ if } ~~ r/2 \le f \le r.
\end{cases}
\end{equation*}
Moreover, we have
\begin{equation}\label{eqn: 0101}
\beta({\bf e}(N)_{d\ell})=(\ell-1) \cdot {\bf e}(M)_d \qa \beta({\bf e}(N)_{d})={\bf e}(M)_d.
\end{equation}
This completes the proof of the first assertion.

Now, suppose that $\ell^4$ does not divide $N$, i.e., $r \le 3$. Then by the computation above it is straightforward that $\b(S_2(N))=S_2(M)$. This proves the last assertion.
\end{proof}

\begin{remark}\label{rmk: another Hecke}
By definition, we have $T_\ell=\beta_\ell(N)_* \circ \alpha_\ell(N)^*$. If $N$ is divisible by $\ell$, i.e., $N=M\ell$ for some integer $M$ as above, then we have
\begin{equation*}
T_\ell = \begin{cases}
\alpha_\ell(M)^* \circ \beta_{\ell}(M)_*-w_\ell & \text{ if } ~~ \gcd(M, \ell)=1, \\
\alpha_\ell(M)^* \circ \beta_{\ell}(M)_* & \text{ otherwise}
\end{cases}
\end{equation*}
(cf. \cite[(2.7)]{Yoo6} or \cite[Rmk. 2.16]{Yoo9}). Thus by (\ref{eqn: 0101}), we easily have $T_{\ell_j}(C_0^{\bde})=0$ for any $t<j\leq u$.
\end{remark}

\begin{remark}\label{rem: new Hecke RCD}
If we define $T_\ell$ as $\alpha_\ell(N)_* \circ \beta_\ell(N)^*$, then we can replace our definition of $C_0^\bde$ appropriately. Also, we can easily prove an analogous result of Lemma \ref{lemma: image of beta} for $\alpha_\ell(M)_*$. 
\end{remark}

\ms
\section{The index of an Eisenstein ideal}\label{section: index}
As in Section \ref{section: notation}, let $N=N_1N_2$ with $\gcd(N_1, N_2)=1$. Also, let
\begin{equation*}
N_1=\prod_{i=1}^t \ell_i \qa N_2=\prod_{j=t+1}^u \ell_i^{r_j}
\end{equation*}
with integers $r_j\ge 2$ for all $t< j \le u$. Let
\begin{equation*}
\fI^N:=(T_\ell-\ell-1 : \ell \nmid N) \subset \T_p=\T \otimes_\Z \Z_{(p)}
\end{equation*}
be an Eisenstein ideal. As in Section \ref{section: strategy}, for any $\bde=(\ve_1, \dots, \ve_t) \in \{\pm 1\}^t$ one may consider an Eisenstein ideal
\begin{equation*}
\fI^\bde=(\fI^N, w_{\ell_i}-\ve_i : 1\leq i \leq t) \subset \T_p
\end{equation*}
and wish to prove that  
\begin{equation*}
\T_p/{\fI^\bde} \simeq \Z_{(p)}/{n \Z_{(p)}}
\end{equation*}
for some integer $n\geq 1$. However, if $N$ is not squarefree, i.e., $u>t$, then the operators $T_{\ell_{r_j}}$ for all $t<j\le u$ are not necessarily congruent to integers modulo $\fI^\bde$. So the argument in Section \ref{section: strategy} does not hold. Nonetheless, if we construct a new Eisenstein ideal $\fI$ containing $\fI^\bde$ so that every operator is congruent to some integer modulo $\fI$, then we can have  $\T_p/{\fI} \simeq \Z_{(p)}/{n \Z_{(p)}}$. 
Thus, it seems natural to consider the following Eisenstein ideal of $\T_p$ as $T_{\ell_j}$ acts as zero on the new subvariety of $J_0(N)$ for all $t<j\le u$: For any $\bde=(\varepsilon_1, \dots, \varepsilon_t) \in \{\pm 1\}^t$, let
\begin{equation*}
\fI_0^\bde:=(\fI^\bde, T_\ell : \ell \mid N_2)=(\fI^\bde, T_{\ell_j} : t+1 \le j \le u)  \qa \fm_0^\bde:=(p, \fI_0^\bde).
\end{equation*}
In this section, we prove the following, which can be regarded as a generalization of Theorem \ref{thm: Ohta index}.
\begin{theorem}\label{thm: index}
Suppose that $N$ is not squarefree, i.e., $N_2>1$. If $p^2$ does not divide $N$, then we have
\begin{equation*}
 \T_p/{\fI_0^{\bde}} \simeq \Z_{(p)}/{\fn_0^{\bde}\Z_{(p)}},
\end{equation*}
where 
\begin{equation*}
\fn_0^{\bde}:=\tn{numerator}\left(\frac{1}{24} \prod_{i=1}^t (\ell_i+\varepsilon_i) \times \prod_{j=t+1}^u \ell_j^{r_j-2}(\ell_j^2-1) \right).
\end{equation*}
\end{theorem}

\ms
\subsection{Generalities}
Throughout this subsection, we always assume that $N$ is not divisible by $p^2$.
We recall some results about modular forms on $X_0(N)$ over $R$, where $R=\Z_{(p)}$ or $R=\zmod {p^m}$ for some $m\geq 1$. 
Our reference is Ohta's paper \cite{Oh14}, and all the Sections, Lemmas, Equations and Propositions below are from there.

Suppose first that $p$ does not divide $N$. As in Section 1.2, let $M_2^A(\Gamma_0(N); R)$ (resp. $S_2^A(\Gamma_0(N); R)$) be the spaces of modular forms (resp. cusp forms) of weight $2$ over $R$ for $\Gamma_0(N)$ in the sense of Deligne and Rapoport \cite{DR73} and Katz \cite{Ka73}. Also, as in Section 1.3, let $M_2(\Gamma_0(N))$ (resp. $S_2(\Gamma_0(N))$) be the complex vector space of modular forms (resp. cusp forms) of weight $2$ for $\Gamma_0(N)$ in the usual sense. For any $f \in M_2(\Gamma_0(N))$, let $f(q)$ denote its $q$-expansion at the cusp $\infty$, and for any $f \in M_2^A(\Gamma_0(N); R)$, let $f(q)$ denote the image of $f$ in $R[[q]]$ as in (1.2.9).
Also, let $a(n; f)$ be the $n$th coefficient of $f(q)$, i.e.,
\begin{equation*}
f(q)=\sum_{n\geq 0} a(n; f) \cdot  q^n.
\end{equation*}
Let 
\begin{equation*}
\begin{split}
M_2^B(\Gamma_0(N); \Z)&:=\{f \in M_2(\Gamma_0(N)) : f(q) \in \Z[[q]] \},\\
S_2^B(\Gamma_0(N); \Z)&:=\{f \in S_2(\Gamma_0(N)) :  f(q) \in \Z[[q]] \}\\
\end{split}
\end{equation*}
and
\begin{equation*}
\begin{split}
M_2^B(\Gamma_0(N); R) &:=M_2^B(\Gamma_0(N); \Z) \otimes_\Z R,\\
S_2^B(\Gamma_0(N); R) &:=S_2^B(\Gamma_0(N); \Z) \otimes_\Z R.
\end{split}
\end{equation*}
By (1.3.4), we have 
\begin{equation*}
M_2^A(\Gamma_0(N); \Z_{(p)})=M_2^B(\Gamma_0(N); \Z_{(p)}).
\end{equation*}
Also, by Lemma (1.3.5), we have 
\begin{equation*}
M_2^B(\Gamma_0(N); \zmod {p^m}) \inj M_2^A(\Gamma_0(N); \zmod {p^m})
\end{equation*}
preserving $q$-expansions as $p$ does not divide $N$.

Next, suppose that $p$ divides $N$. By our assumption, we have $p \dmid N$.
As in Section 1.4, let
\begin{equation*}
\begin{split}
M_2^{\tn{reg}}(\Gamma_0(N); R)&:=H^0(X_0(N)_{/R}, \Omega(\textit{cusps})),\\
S_2^{\tn{reg}}(\Gamma_0(N); R)&:=H^0(X_0(N)_{/R}, \Omega),\\
\end{split}
\end{equation*}
where $\Omega$ is the sheaf of regular differentials on $X_0(N)_{/R}$ and $\textit{cusps}$ is the scheme of cusps of $X_0(N)_{/R}$. 
By Proposition (1.4.8), there there exist $q$-expansion preserving maps:
\begin{equation*}
\begin{split}
M_2^{\tn{reg}}(\Gamma_0(N); R) \to M_2^B(\Gamma_0(N); R),\\
S_2^{\tn{reg}}(\Gamma_0(N); R) \to S_2^B(\Gamma_0(N); R),\\
\end{split}
\end{equation*}
which are injections when $R=\Z_{(p)}$. 
If $p$ does not divide $N$, then $R$ is a $\Z[1/N]$-algebra, and so the above maps are in fact isomorphisms (cf. Corollary (1.4.10)). We henceforth follow Convention (1.4.16), and write $f(q)$ for the $q$-expansion of $f \in M_2^{\reg}(\Gamma_0(N); R)$. 

By the discussions in Section 1.5, there are actions of the Hecke operators $T_\ell$ (for any primes $\ell \neq p$) and the Atkin--Lehner operators $w_\ell$ (for any primes $\ell \mid N_1$) on $M_2^{\tn{reg}}(\Gamma_0(N); R)$ (resp. $S_2^{\tn{reg}}(\Gamma_0(N); R)$). As usual, we identify $\T_p$ with a subring of $\End(S_2^{\tn{reg}}(\Gamma_0(N); \Z_{(p)}))$ generated by the same-named operators (cf. \cite[Sec. 1.1]{MR91} or \cite[pg. 275]{Oh14}). Let
\begin{equation*}
\begin{split}
M_2^{\tn{reg}}(\Gamma_0(N); R)^{\bde}&:=M_2^{\tn{reg}}(\Gamma_0(N); R)[w_{\ell_i}-\varepsilon_i : 1\leq i \leq t],\\
S_2^{\tn{reg}}(\Gamma_0(N); R)^{\bde}&:=S_2^{\tn{reg}}(\Gamma_0(N); R)[w_{\ell_i}-\varepsilon_i : 1\leq i \leq t].
\end{split}
\end{equation*}

The following is crucial in our proof.
\begin{proposition}\label{prop: g=0 when}
Let $R=\zmod {p^m}$ for some $m\geq 1$. Let 
\begin{equation*}
g(q)=\sum_{n\geq 0} a(n; g) \cdot q^n \in M_2^{\tn{reg}}(\Gamma_0(N); R)^{\bde}
\end{equation*}
such that $a(n; g)=0$ unless $\gcd(n, N_1)>1$. Suppose further that $a(0; g)=0$. Then $g=0$.
\end{proposition}
\begin{proof}
Suppose first that $p$ does not divide $N$. By the discussion above, we have
\begin{equation*}
M_2^{\tn{reg}}(\Gamma_0(N); R) =M_2^B(\Gamma_0(N); R) \inj M_2^A(\Gamma_0(N); R)
\end{equation*}
preserving $q$-expansions. Thus, $g$ can be regarded as an element of $M_2^A(\Gamma_0(N); R)$ such that $w_{\ell_i}(g)=\varepsilon_i \cdot g$ for all $1\leq i\leq t$ and $a(n; g)=0$ unless $\gcd(n, N_1)>1$. So $g$ is a constant by Proposition (2.1.2). Since $a(0; g)=0$, we have $g=0$, as claimed.

Suppose next that $p$ divides $N$. By our assumption, we have $p \dmid N$. Suppose that $g\neq 0$, i.e., there is an integer $k$ such that $a(k; g) \not\equiv 0 \pmod {p^m}$.
Let $\pi : \Z_{(p)} \to R$ be the natural map induced by the reduction modulo $p^m$. Since the map 
\begin{equation*}
M_2^{\tn{reg}}(\Gamma_0(N); \Z_{(p)}) \to M_2^{\tn{reg}}(\Gamma_0(N); R) 
\end{equation*}
induced by $\pi$ is surjective, there is a lift 
\begin{equation*}
\textstyle\wt{g}(q)=\sum_{n\geq 0} a(n; \wt{g}) \cdot q^n \in M_2^{\tn{reg}}(\Gamma_0(N); \Z_{(p)})
\end{equation*}
such that $\pi(a(n; \wt{g})) =a(n; g)$ for all $n\geq 0$.  Since $a(k; g) \not\equiv 0 \pmod {p^m}$, the greatest common divisor of the coefficients $a(n; \wt{g})$ is $p^a$ for some $0\leq a<m$. By definition, we have $p^{-a} \cdot a(n; g) \in \Z_{(p)}$ for all $n\geq 0$, and so 
\begin{equation*}
f:=p^{-a} \cdot \wt{g} \in M_2^B(\Gamma_0(N); \Z_{(p)}).
\end{equation*}
Let $w_p(\wt{g})(q)=\sum_{n\geq 0} a_n \cdot q^n$. 
Since $w_p(g)=\varepsilon g$ for some $\varepsilon \in \{ \pm 1\}$, we have $a_n \equiv \varepsilon  a(n; g) \pmod {p^m}$, and so
$a_n$ is divisible by $p^a$ for all $n\geq 0$. Thus, $f \in M_2^{\tn{reg}}(\Gamma_0(N); \Z_{(p)})$ by Proposition (1.4.9).
Let $\ov{f}$ be the image of $f$ in $M_2^{\tn{reg}}(\Gamma_0(N); \F_p)$, which is non-zero by its construction.
Since $a<m$, we have $w_{\ell_i}(\ov{f})=\varepsilon_i \cdot \ov{f}$ for all $1\leq i \leq t$ and
$a(n; f)=0$ unless $\gcd(n, N_1)>1$. Thus, by the same argument as in Proposition (2.2.7), we have $\ov{f}=0$, which is a contradiction. (Indeed, Ohta assumed that $N$ is squarefree in the proposition, but it is not used in the proof. What is used there is that $\ell_i \dmid N$ for all $1\leq i \leq s$, which are guaranteed here.)
Therefore we have $g=0$, as claimed.
\end{proof}

Since we have Proposition \ref{prop: g=0 when}, we can mimic the arguments in Theorem (2.4.6) and Corollary (2.4.7). As a result, we have the following.
\begin{proposition}\label{prop: duality}
There is a perfect pairing:
\begin{equation*}
S_2^{\tn{reg}}(\Gamma_0(N); \Z_{(p)})^{\bde} \times \T_p^{\bde} \overset{( ~~,~~ )}{\longrightarrow} \Z_{(p)},
\end{equation*}
where $\T_p^{\bde}:=\T_p/{\br{w_{\ell_i}-\varepsilon_i : 1\leq i \leq t}}$.
\end{proposition}

\ms
\subsection{Eisenstein series} 
In this subsection, we construct an Eisenstein series $E_0^{\bde}$ using various ``level-raising'' maps between the spaces of modular forms (cf. \cite[Sec. 2.3]{Oh14}, \cite[Def. 2.5]{Yoo1}, \cite[Sec. 5]{Yoo6}). 
If $N$ is squarefree, then all the results in this subsection are already discussed in \cite{Oh14}, and therefore we assume that $N$ is not squarefree, i.e., $u>t$.
\begin{definition}
Let $\ell$ be a prime not dividing $M$. We define the maps $[\ell]^+$ and $[\ell]^-$ from $M_2(\Gamma_0(M))$ to $M_2(\Gamma_0(M\ell))$ by 
\begin{equation*}
[\ell]^+(f)(\tau):=f(\tau)+\ell f(\ell \tau) \qa [\ell]^-(f)(\tau):=f(\tau)-\ell f(\ell \tau).
\end{equation*}
For any $r\geq 2$, we define a map $[\ell^r]^0$ from $M_2(\Gamma_0(M))$ to $M_2(\Gamma_0(M\ell^r))$ by
\begin{equation*}
[\ell^r]^0(f)(\tau):=f(\tau)-(\ell+1)f(\ell \tau)+\ell f(\ell^2 \tau).
\end{equation*}
In fact, $[\ell^r]^0(f)$ can be regarded as a modular form for $\Gamma_0(M\ell^2)$.
\end{definition}
Note that for any $f \in M_2(\Gamma_0(M))$ and $r\geq 2$, we have
\begin{equation*}
a(1; [\ell^r]^0(f))=a(1; [\ell]^+(f))=a(1; [\ell]^-(f))=a(1; f).
\end{equation*}

\begin{remark}\label{rem: another}
In fact, the maps above are constructed by the degeneracy maps as follows:
\begin{equation*}
\begin{split}
[\ell]^\pm (f)&=\left(\alpha_\ell(M)^* \pm \beta_\ell(M)^*\right)(f), \\
[\ell^2]^0(f)&=\left( \alpha_\ell(M\ell)^* -\frac{1}{\ell} \beta_\ell(M\ell)^* \right) \circ [\ell]^-(f),\\
[\ell^r]^0(f)&=\alpha_\ell(M\ell^{r-1})^* \circ  \cdots \circ \alpha_\ell(M\ell^2)^* \circ [\ell^2]^0(f) \quad \text{ for any $r\ge 3$.}
\end{split}
\end{equation*}
\end{remark}

\begin{definition}
We define 
\begin{equation*}
E_0^{\bde}:= [\ell_1]^{\varepsilon_1} \circ \cdots \circ [\ell_t]^{\varepsilon_t} \circ [\ell_{t+1}^{r_{t+1}}]^0 \circ \cdots \circ [\ell_u^{r_u}]^0 \circ (K) \in M_2(\Gamma_0(N)),
\end{equation*}
where $K$\textemdash see \cite[Sec. 2.3]{Oh14} for its definition\textemdash is a non-holomorphic Eisenstein series of level $1$. Since $K$ is not a genuine modular form, the definition seems invalid. However, $[\ell]^-(K)$ is a genuine modular form for $\Gamma_0(\ell)$ as the non-holomorphic term of $[\ell]^-(K)$ vanishes. Thus, if either $N$ is not squarefree (in which case $[\ell^r]^0$ is the composition of other maps and $[\ell]^-$) or $\bde \ne \bde_+$, then this definition makes sense.
\end{definition}

By its construction, we easily have the following.
\begin{lemma}\label{lemma: Eisenstein series Hecke action}
For any $1\leq i \leq t < j \le u$ and any prime $\ell$ not dividing $N$, we have 
\begin{equation*}
w_{\ell_i}(E_0^{\bde})=\varepsilon_i \cdot E_0^{\bde},  \quad T_{\ell_j}(E_0^\bde)=0 \qa T_\ell(E_0^{\bde})=(\ell+1) \cdot E_0^{\bde}.
\end{equation*}
Also, we have $a(1; E_0^{\bde})=1$. 
\end{lemma}
\begin{proof}
By the same argument as in \cite[Prop. 5.4]{Yoo6}, the first assertion follows by Remark \ref{rem: another}. Since $a(1; K)=1$, the second assertion follows by the definition of the ``level-raising'' maps.
\end{proof}

\begin{remark}\label{rem: new Hecke Eisenstein}
If we define $T_\ell$ as $\a_\ell(N)_* \circ \b_\ell(N)^*$, then we can replace the operator $[\ell^r]^0$ so that $E_0^\bde$ is annihilated by $T_{\ell_j}$ for all $t<j\le u$. Then all the remaining arguments in this paper can be easily modified. 
\end{remark}

\begin{lemma}\label{lemma: residue}
The residue of $E_0^{\bde}$ at any cusp of level $d$ is $0$ if $d$ is not squarefree, and
\begin{equation*}
\frac{1}{24}\prod_{i=1}^t  \varepsilon_i^{f_i}(\ell_i+\varepsilon_i)\prod_{j=t+1}^u \ell_j^{r_j-3}(1-\ell_j^2)(1-\ell_j)^{1-f_j} 
\end{equation*}
if $d=\prod_{i=1}^u \ell_i^{f_i}$ is squarefree. In particular, the residue of $E_0^{\bde}$ at any cusp of level $\prod_{i=1}^u \ell_i$ is $\pm\frac{\fn_0^{\bde}}{\prod_{j=t+1}^u \ell_j}$.
\end{lemma}
\begin{proof}
We use the same argument as in \cite[Sec. 5.2]{Yoo6}. 

Since $E_0^\bde$ is constructed by the degeneracy maps (Remark \ref{rem: another}), the residues of $E_0^\bde$ at all cusps of level $d$ are the same by Lemmas 5.6 and 5.7 of \textit{loc. cit.} So it suffices to compute the residue of $E_0^\bde$ at a cusp of level $d$ for any divisor $d$ of $N$.

We prove the assertion by induction on $u$. First, let $N=\ell^r$ with $r\ge 1$ and $d=\ell^f$ for some $0\le f \le r$.
By \cite[Lem. 2.1]{Yoo6}, the ramification index of $\a_\ell(1)$ at a cusp of level $1$ (resp. $\ell$) is $\ell$ (resp. $1$).
Also, the ramification index of $\b_\ell(1)$ at a cusp of level $1$ (resp. $\ell$) is $1$ (resp. $\ell$).
Thus, by \cite[Lem. 5.6]{Yoo6} the residue of $[\ell]^-(K)$ at a cusp of level $1$ (resp. $\ell$)  is $\frac{\ell-1}{24}$ (resp. $\frac{1-\ell}{24}$).
Similarly, by \cite[Lems. 2.1 and 5.6]{Yoo6} the residue of $[\ell^2]^0(K)$ at a cusp of level $1$ is
\begin{equation*}
\frac{\ell(\ell-1)}{24}-\frac{(\ell-1)}{24\ell}=\frac{1}{24\ell}(1-\ell^2)(1-\ell).
\end{equation*}
Also, the residue of $[\ell^2]^0(K)$ at a cusp of level $\ell$ is
\begin{equation*}
\frac{1-\ell}{24}-\frac{\ell-1}{24\ell}=\frac{1-\ell^2}{24\ell}.
\end{equation*}
Furthermore, the residue of $[\ell^2]^0(K)$ at a cusp of level $\ell^2$ is 
\begin{equation*}
\frac{1-\ell}{24}-\frac{1-\ell}{24}=0.
\end{equation*}
This computation proves the assertion for $r\le 2$. Suppose that $r\ge 3$. Let
\begin{equation*}
\a=\alpha_\ell(\ell^2) \circ \cdots \circ \alpha_\ell(\ell^{r-1}) : X_0(\ell^2) \to X_0(\ell^r).
\end{equation*}
By \cite[Lem. 2.1]{Yoo6}, $\a$ is ramified at a cusp of level $1$ (resp. level $\ell$) of ramification index $\ell^{r-2}$. Thus by \cite[Lem. 5.6]{Yoo6}, the residues of $E_0^\bde=\a^*([\ell^2]^0(K))$ at cusps of level $1$ and $\ell$ are 
\begin{equation*}
\frac{\ell^{r-3}}{24}(1-\ell^2)(1-\ell) \qa \frac{\ell^{r-3}}{24}(1-\ell^2),
\end{equation*}
respectively. Moreover, for any $f\ge 2$ the image of a cusp of level $\ell^f$ in $X_0(\ell^r)$ by $\a$ is $\vect 1 {\ell^2}$. 
Thus, the residue of $E_0^\bde$ at a cusp of level $\ell^f$ with $f\ge 2$ is $0$. This proves the assertion for $u=1$.

Next, suppose that $u\ge 1$ and $d=\prod_{i=1}^u \ell_i^{f_i}$ is a divisor of $N$.
We divide into two cases.
\begin{enumerate}
\item
Suppose that $\ell=\ell_i$ for some $1\le i \le t$. Let $\ve=\ve_i \in \{ \pm 1\}$ and $f=f_i \in \{0, 1\}$. Also, let $M=N/\ell$ and $\d=\gcd(M, d)$. 
Furthermore, let $\bde'=(\ve_1, \dots, \widehat{\ve_i}, \dots, \ve_t) \in \{\pm 1\}^{t-1}$.
Then we can construct an Eisenstein series $E_0^{\bde'}$ of level $M$, which is denoted by $E$.
Since the ramification index of $\a_\ell(M)$ (resp. $\b_\ell(M)$) at a cusp of level $d$ is $\ell^{1-f}$ (resp. $\ell^f$), 
the residue of $E_0^\bde=[\ell]^{\ve} (E)$ at a cusp of level $d$ is 
\begin{equation*}
(\ell^{1-f} +\ve \ell^f)a=\ve^f(\ell+\ve)a,
\end{equation*}
where $a$ is the residue of $E$ at a cusp of level $\d$.
\item
Suppose that $\ell=\ell_j$ for some $t< j \le u$. Let $f=f_j$ and $r=r_j\ge 2$. Also, let $M=N\ell^{-r}$ and $\d=\gcd(M, d)$. Then as above, we can construct an Eisenstein series $E_0^\bde$ of level $M$, which is denoted by $E$.
Let $a$ be the residue of $E$ at a cusp of level $\d$ in $X_0(M)$. 
By the same argument as above, the residue of $[\ell]^-(E)$ at a cusp of level $\d$ (resp. $\d\ell$) is $(\ell-1)a$ (resp. $(1-\ell)a$). 
Again, the residue of $\a_\ell(M\ell)^*([\ell]^-(E))$ at a cusp of level $\d$ (resp. $\d \ell$) is $\ell(\ell-1)a$ (resp. $(1-\ell)a$). Similarly, the residue of $\b_\ell(M\ell)^*([\ell]^-(E))$ at a cusp of level $\d$ or $\d\ell$ is $(\ell-1)a$.
Thus, the residue of $[\ell^2]^0(E)$ at a cusp of level $\d$ is
\begin{equation*}
\ell(\ell-1)a-\frac{1}{\ell}(\ell-1)a=\frac{1}{\ell}(1-\ell^2)(1-\ell)a.
\end{equation*}
Similarly, the residue of $[\ell^2]^0(E)$ at a cusp of level $\d\ell$ is
\begin{equation*}
(1-\ell)a-\frac{1}{\ell}(\ell-1)a=\frac{1}{\ell}(1-\ell)(\ell+1)a=\frac{1}{\ell}(1-\ell^2)a.
\end{equation*}
Furthermore, the residue of $[\ell^2]^0(E)$ at a cusp of level $\d\ell^2$ is $0$. Since $\a=\alpha_\ell(M\ell^2) \circ \cdots \circ \alpha_\ell(M\ell^{r-1})$ is ramified at a cusp of level $\d$ or $\d\ell$ of ramification index $\ell^{r-2}$, the residue of $E_0^\bde=\a^*\circ [\ell^2]^0(E)$ at a cusp of level $d=\d \ell^f$ for $f \in \{0, 1\}$ (resp. for $f\ge 2$) is
$\ell^{r-3}(1-\ell^2)(1-\ell)^{1-f}a$ (resp. $0$). 
\end{enumerate}
By induction, the assertion follows.
\end{proof}

\ms
\subsection{Proof of Theorem \ref{thm: index}}
Suppose that $\T_p /{\fI_0^{\bde}} \simeq \zmod {p^m}$ for some $m\geq 0$.  Note that we have
\begin{equation*}
\T_p/{\fI_0^{\bde}} \surj \End(\br{[C_0^{\bde}]}[p^\infty]) \simeq \Z_{(p)}/{\fn_0^{\bde}\Z_{(p)}},
\end{equation*}
where the first surjection follows by Lemma \ref{lemma: C annihilated by I} and the second isomorphism follows by Lemma \ref{lemma: order}. Thus, we have
$m\geq \tn{val}_p(\fn_0^{\bde})$, and it suffices to show that $m\leq \tn{val}_p(\fn_0^{\bde})$. 
If $m=0$, then there is nothing to prove, so we assume that $m\geq 1$. 

We first claim that $E_0^{\bde} \in M_2^{\tn{reg}}(\Gamma_0(N); \Z_{(p)})^{\bde}$. For ease of notation, let
\begin{equation*}
M_2^{\tn{reg}}:=M_2^{\tn{reg}}(\Gamma_0(N); \Z_{(p)}) \qa M_2^B:=M_2^B(\Gamma_0(N); \Z_{(p)}).
\end{equation*}
By Lemma \ref{lemma: Eisenstein series Hecke action}, it suffices to show that $E_0^{\bde} \in M_2^{\tn{reg}}$.
Since $N$ is not squarefree, the residue of $E_0^\bde$ at the cusp $\infty$, which is of level $N$, is zero by Lemma
\ref{lemma: residue}. Thus, we have $a(0; E_0^\bde)=0$. (We can obtain this by the definition of $[\ell^r]^0$ as well.)
Also, for any $n\ge 1$, we have $a(n; E_0^\bde) \in \Z$ by the definition as $a(n; K) \in \Z$.
Thus, we have $E_0^\bde \in M_2^B(\G_0(N); \Z)$. Now, if $p$ does not divide $N$, then the claim follows as $M_2^B(\G_0(N); \Z)\subset M_2^B=M_2^{\tn{reg}}$. 
If $p \dmid N$, then the claim follows by Proposition (1.4.9) of \cite{Oh14} as $w_p(E_0^{\bde})=\pm E_0^{\bde}$.

Next, we note that $\T_p/{\fI_0^{\bde}}=\T_p^{\bde}/{\fI_0}$. So by Proposition \ref{prop: duality}, there is 
\begin{equation*}
f(q)=\textstyle\sum_{n\geq 1} a(n; f) \cdot q^n \in S_2^{\tn{reg}}(\Gamma_0(N); \zmod {p^m})^{\bde}
\end{equation*}
such that $a(n; f)=T_n \pmod {\fI_0} \in \T_p^{\bde}/{\fI_0} \simeq \zmod {p^m}$
whenever $\gcd(n, N_1)=1$.

Finally, we consider the image of $E_0^{\bde}$ in $M_2^{\tn{reg}}(\Gamma_0(N); \zmod {p^m})^{\bde}$, denoted by $E$. Note that $a(0; E)=0$ as $a(0; E_0^\bde)=0$. Thus, if we take 
\begin{equation*}
 g=E-f \in M_2^{\tn{reg}}(\Gamma_0(N); \zmod {p^m})^{\bde},
 \end{equation*}
then it satisfies all the assumptions in Proposition \ref{prop: g=0 when}. Thus, we have $E=f$. Since $f \in S_2^{\tn{reg}}(\Gamma_0(N); \zmod {p^m})^{\bde}$, the residue of $E_0^{\bde}$ at any cusp must be divisible by $p^m$. By Lemma \ref{lemma: residue}, we then have 
\begin{equation*}
m \leq \tn{val}_p\left(\frac{\fn_0^{\bde}}{\prod_{j=t+1}^u \ell_j} \right)=\tn{val}_p(\fn_0^{\bde})
\end{equation*}
as $p$ does not divide $N_2$. This completes the proof. \qed

\ms
\section{The kernel of an Eisenstein ideal in characteristic $p$}\label{section: multi}
As in the previous section, let $\bde=(\varepsilon_1, \dots, \varepsilon_t) \in \{\pm 1\}^t$.
In this section, we prove the following, which can be regarded as a generalization of Theorem \ref{thm: Ohta multi}.

\begin{theorem}\label{thm: multi one}
Suppose that $N$ is not squarefree, i.e., $N_2>1$. Also, suppose that $p^2$ does not divide $N$. Then $\cB_p[\fI_0^\bde]$ is cyclic as a $\T_p/{\fI_0^\bde}$-module.
\end{theorem}

\begin{proof}
As in the proof of Theorem \ref{thm: Ohta multi}, we may assume that $\fm_0^\bde$ is maximal and 
it suffices to show that the dimension of $\cJ_p(\F_p)[\fm_0^\bde]$ over $\T_p/{\fm_0^\bde} \simeq \F_p$ is at most $1$, where $\cJ_p$ is the special fiber of the N\'eron model of $J_0(N)$ at $p$.

We divide into two cases: 
\begin{enumerate}
\item \label{eqnn1}
Suppose that $p$ does not divide $N$ and so $\cJ_p=J_0(N)_{/\F_p}$ as $J_0(N)$ has good reduction at $p$.
Then the result follows by Mazur's argument \cite[Ch. II, Cor. 14.8]{M77}. For instance, see Ohta's argument in \cite[Prop. 3.5.4]{Oh14}, which is a direct generalization of Mazur's to squarefree level. It works verbatim for non-squarefree $N$.

\item  \label{eqnn2}
Suppose that $N=Mp$ with $\gcd(M, p)=1$ and $p\geq 3$. By Deligne--Rapoport \cite{DR73}, $X_0(N)_{\F_p}$ consists of two copies of $X_0(M)_{/\F_p}$ which meet transversally at supersingular points. By the theory of Picard functor by Raynaud \cite{Ra70} (which is again explained in \cite[Ch. 9]{BLR90}), there is a Hecke-equivariant exact sequence:
\begin{equation}\label{eqn1}
\xymatrix{
0 \ar[r] & \cJ_p^0 \ar[r] & \cJ_p \ar[r] & \Phi_p \ar[r] & 0,
}
\end{equation}
where $\cJ_p^0$ is the identity component and $\Phi_p$ is the component group. Moreover,  there is also a Hecke-equivariant exact sequence:
\begin{equation}\label{eqn2}
\xymatrix{
0 \ar[r] & \cT_p \ar[r] & \cJ_p^0 \ar[r] & J_0(M)_{/\F_p} \times J_0(M)_{/\F_p} \ar[r] & 0,
}
\end{equation}
where $\cT_p$ is the torus of $\cJ_p$. 
To prove the assertion, we compute the kernels of $\fm_0^\bde$ on the $\F_p$-points of $\cT_p$, $\Phi_p^0$ and $J_0(M)_{/\F_p} \times J_0(M)_{/\F_p}$, respectively.
\begin{enumerate}
\item  \label{eqnna}
$\cT_p(\F_p)[\fm_0^{\bde}]=0$ as $\cT_p(\F_p)$ does not have non-trivial $p$-torsion.
\item  \label{eqnnb}
Let $\ell$ be a prime divisor of $N_2$, i.e., $\ell^2$ divides $N$. Thus, by our assumption we have $\ell \ne p$. Then $\Phi_p(\F_p)[\fm_0^{\bde}] \subset \Phi_p[p, T_{\ell}]=0$ as  $T_{\ell}$ acts on $\Phi_p$ by $\ell$ (cf. \cite[Th. 1.1]{KY18}).
\item  \label{eqnnc}
We rearrange indices so that $p=\ell_1$. For avoidance of confusion, for a prime $\ell$, let $\tau_\ell$ be the $\ell$th Hecke operator of level $M$ acting on $J_0(M)$, and for any $2\leq i \leq t$, let $\omega_{i}$ be the Atkin--Lehner operator with respect to $\ell_i$ acting on $J_0(M)$. Let
\begin{equation*}
\T':=\Z[\tau_\ell, \omega_{i} : \text{ for any primes } \ell \nmid N_1 \qqa 2\leq i \leq t] \subset \End(J_0(M)).
\end{equation*}

As a variant of Ribet's lemma \cite[pg. 491]{Wi95}, we have $\T'=\T(M)$. (Note that $\T(M)$ contains the $p$th Hecke operator $\tau_p$, but $\T'$ does not. Nonetheless, the claim follows as $p$ is odd.) Let
\begin{equation*}
\fI'=(\omega_{i}-\epsilon_i, \tau_{\ell_j}, \tau_\ell-\ell-1 : \text{ for any primes } \ell \nmid N \qqa 2\leq i \leq t < j \le u)
\end{equation*}
be the corresponding ideal of $\fI_0^\bde$ in $\T'$. 
Then we can construct a rational cuspidal divisor $C$ on $X_0(M)$ annihilated by $\fI'$.
Indeed, if we let $\bde':=(\varepsilon_2, \dots, \varepsilon_t) \in \{\pm 1\}^{t-1}$, then we set $C:=C_0^{\bde'} \in \Qdiv M$, which is of degree $0$ as $M$ is not squarefree (see Definition \ref{def: RCD}).
By Lemma \ref{lemma: order}, the order of the linear equivalence class of $C$ in $J_0(M)$ is either $\frac{\fn_0^{\bde}}{(p+\varepsilon_1)}$ or $\frac{2\fn_0^{\bde}}{(p+\varepsilon_1)}$. Thus, its $p$-part is equal to that of $\fn_0^{\bde}$, and so the ideal $\fm':=(p, \fI_0^{\bde'}) \subset \T(M)$ is also maximal. 
Since $p$ does not divide $M$, by the same argument as in \eqref{eqnn1} above we have
\begin{equation*}
\dim_{\F_p} J_0(M)_{/\F_p}(\F_p)[\fm'] \leq 1.
\end{equation*}
Note that $\T(N)$ acts on $J_0(M)_{/\F_p} \times J_0(M)_{/\F_p}$ as follows: For any primes $\ell$ not dividing $N_1$ and $2\leq i \leq t$, we have
\begin{equation*}
w_p=\mat 0 1 1 0, \quad T_\ell=\mat {\tau_\ell} 0 0 {\tau_\ell} \qa w_{\ell_i}=\mat {\omega_{i}} 0 0 {\omega_{i}}.
\end{equation*}
(For instance, see the proof of \cite[Prop. 3.5.9]{Oh14}.) Thus, there is an isomorphism
\begin{equation*}
J_0(M)_{/\F_p}(\F_p)[\fm'] \simeq (J_0(M)_{/\F_p} \times J_0(M)_{/\F_p})(\F_p)[\fm_0^{\bde}]
\end{equation*}
sending $x$ to $(x, \varepsilon_1 x)$, and so the dimension of $(J_0(M)_{/\F_p} \times J_0(M)_{/\F_p})(\F_p)[\fm_0^{\bde}]$ is at most $1$ .
\end{enumerate}

In conclusion, by \eqref{eqnnb} and \eqref{eqn1} we have $\cJ_p(\F_p)[\fm_0^\bde]=\cJ_p^0(\F_p)[\fm_0^\bde]$. Also, by \eqref{eqnna} and \eqref{eqn2} we have
$\cJ_p^0(\F_p)[\fm_0^\bde]=(J_0(M)_{/\F_p} \times J_0(M)_{/\F_p})(\F_p)[\fm_0^{\bde}]$. Thus, the assertion follows.
\end{enumerate}
This completes the proof.
\end{proof}

\begin{remark}
The proof of \eqref{eqnn2} above is almost the same as that of \cite[Prop. 3.5.9]{Oh14}. More specifically, the only difference is the component group part \eqref{eqnnb}. 
In our case, \eqref{eqnnb} is much simpler thanks to the existence of the prime divisor $\ell$ of $N_2$. For instance, it is not necessary to distinguish the case $\bde'=\bde_+'$ with others (or to require that $\Phi_3$ does not have extra elementary $3$-groups). 
\end{remark}

\ms
\section{Proof of the main theorem}\label{section: proof}
In this section, we prove Theorem \ref{theorem : main theorem}. We use the same notation as in the previous sections. In particular, we denote by
\begin{equation*}
\cA_p=\scC(N)[p^\infty] \qa  \cB_p=J_0(N)(\Q)_\tor[p^\infty].
\end{equation*}
During the section, we assume that $p$ is an odd prime whose square does not divide $N$.
If $p=3$ (so $N$ is not divisible by $9$), then we further assume that there is a prime divisor of $N$ congruent to $-1$ modulo $3$. 

\ms
First, by our assumption on $p$ we can use Theorems \ref{thm: index} and \ref{thm: multi one}. 
Thus, by the same argument as in Section \ref{section: strategy}, we have $\# \cB_p[\fI_0^\bde] \le p^{\val_p(\fn_0^\bde)}$ for any $\bde \in \{\pm 1\}^\bde$.
Since $\cA_p[\fI_0^\bde]$ contains $\br{[C_0^\bde]}[p^\infty]$ by Lemma \ref{lemma: C annihilated by I}, 
we have $\# \cA_p[\fI_0^\bde] \ge \br{[C_0^\bde]}[p^\infty]=p^{\val_p(\fn_0^\bde)}$ by Lemma \ref{lemma: order}. 
Since we have $\cA_p[\fI_0^\bde] \subset \cB_p[\fI_0^\bde]$, for any $\bde \in \{\pm 1\}^\bde$ we have  
\begin{equation}\label{eqn7575}
\cA_p[\fI_0^\bde]=\cB_p[\fI_0^\bde].
\end{equation}
 
Next, we consider the following Eisenstein ideal of $\T_p$:
\begin{equation*}
\fI_0 :=(T_{\ell_j}, \fI^N : t+1 \le j \le u).
\end{equation*}
By the same argument as in Lemma \ref{lemma: decomposition}, we have
\begin{equation*}
\T_p/{\fI_0} \simeq \prod_{\bde \in \{\pm 1\}^t} \T_p/{\fI_0^\bde}.
\end{equation*}
Thus, as $\T_p/{\fI_0}$-modules we have the following decompositions:
\begin{equation*}
\cA_p[\fI_0]  \simeq \moplus_{\bde \in \{\pm 1\}^t} \cA_p[\fI_0^\bde] \qa \cB_p[\fI_0]  \simeq \moplus_{\bde \in \{\pm 1\}^t} \cB_p[\fI_0^\bde].
\end{equation*}
By \eqref{eqn7575} we have
\begin{equation}\label{eqn22}
\cA_p[\fI_0]=\cB_p[\fI_0].
\end{equation}

Finally, by Lemma  \ref{lemma: ES relation} we have $\cA_p=\cA_p[\fI^N]$ and $\cB_p=\cB_p[\fI^N]$.
Similarly, we have
\begin{equation*}
 \cA_p[\fI_0]=\cA_p[T_{\ell} : \ell \mid N_2] \qa \cB_p[\fI_0]=\cB_p[T_{\ell} : \ell \mid N_2].
\end{equation*}
Thus, it suffices to prove the following implication:
\begin{equation*}
\cA_p[T_{\ell} : \ell \mid N_2]=\cB_p[T_{\ell} : \ell \mid N_2] \imply \cA_p=\cB_p.
\end{equation*}
One of the key contributions of this paper is the following inductive argument.
\begin{theorem}\label{thm: inductive method}
Suppose that $\scC(N/{\ell})[p^\infty]=J_0(N/{\ell})(\Q)_\tor[p^\infty]$ for all prime divisors $\ell$ of $N_2$. Then we have
\begin{equation*}
\cA_p[T_{\ell} : \ell \mid N_2]=\cB_p[T_{\ell} : \ell \mid N_2] \imply \cA_p=\cB_p.
\end{equation*}
\end{theorem}

\begin{proof}
Suppose that $\ell$ is a prime divisor of $N_2$, i.e., $\ell^2$ divides $N$. By our assumption on $p$, we also have $\ell \ne p$. 
Since $T_\ell$ preserves $\cA_p$ and $\cB_p$, we have the following commutative diagram, where the two rows are exact and
the vertical arrows are induced by the natural inclusion $\scC(N) \subseteq J_0(N)(\Q)_\tor$:
\begin{equation*}
\xymatrix{
0 \ar[r] & \cB_p[T_\ell] \ar[r] & \cB_p[\fI^N] \ar[r]^-{T_\ell} & T_\ell(\cB_p) \ar[r] & 0 \\
0 \ar[r] & \ar[u] \cA_p[T_\ell] \ar[r] & \ar[u]  \cA_p[\fI^N] \ar[r]^-{T_\ell} & \ar[u]  T_\ell(\cA_p) \ar[r] & 0.
}
\end{equation*}
Since all the vertical arrows are injective, by five lemma, $\cA_p=\cB_p$ if $\cA_p[T_\ell]=\cB_p[T_\ell]$ and $T_\ell(\cA_p)=T_\ell(\cB_p)$.

Let $M=N/\ell$, and let
\begin{equation*}
\alpha:=\alpha_\ell(M) \qa \beta:=\beta_\ell(M)
\end{equation*}
denote two degeneracy maps from $X_0(N)$ to $X_0(M)$. 
Since the degeneracy map $\beta$ is defined over $\Q$, we have
\begin{equation*}
\beta_{*}(\cB_p) \subseteq J_0(M)(\Q)_\tor[p^\infty]=\scC(M)[p^\infty],
\end{equation*}
where the equality follows by our induction hypothesis. Although the definition of $T_\ell$ is given by $\beta_\ell(N)_* \circ \alpha(N)^*$ (Section \ref{section: notation}),  we have $T_\ell=\alpha^* \circ \beta_{*}$ as endomorphisms of $J_0(N)$ as $\ell^2$ divides $N$ (Remark \ref{rmk: another Hecke}). 
Since we have $\beta_{*}(\cA_p)=\scC(M)[p^\infty]$ by Lemma \ref{lemma: image of beta}, we have
\begin{equation*}
T_\ell(\cB_p) =\alpha^*(\beta_{*}(\cB_p))\subseteq \alpha^*(\scC(M)[p^\infty])=\alpha^*(\beta_{*}(\cA_p))=T_\ell(\cA_p).
\end{equation*}
Note that $\cA_p \subseteq \cB_p$ and so we have $T_\ell(\cA_p)=T_\ell(\cB_p)$. 
Hence we have $\cA_p=\cB_p$ if $\cA_p[T_\ell] = \cB_p[T_\ell]$.

Let $\ell'$ be a prime divisor of $N_2$ different from $\ell$. Then by the same argument as above, $\cA_p=\cB_p$ if $\cA_p[T_{\ell'}]=\cB_p[T_{\ell'}]$.
Since the operators $T_\ell$ and $T_{\ell'}$ commute with each other, $T_\ell$ preserves $\cA_p[T_{\ell'}]$ and $\cB_p[T_{\ell'}]$, and moreover  we have
\begin{equation*}
T_\ell(\cA_p[T_{\ell'}])=T_\ell(\cA_p)[T_{\ell'}]=T_\ell(\cB_p)[T_{\ell'}]=T_\ell(\cB_p[T_{\ell'}])
\end{equation*}
as $T_\ell(\cA_p)=T_\ell(\cB_p)$. Thus, if we replace $\cA_p$ and $\cB_p$ by $\cA_p[T_{\ell'}]$ and $\cB_p[T_{\ell'}]$, respectively, then we have the following:
\begin{equation*}
\cA_p[T_{\ell'}][T_\ell] = \cB_p[T_{\ell'}][T_\ell] \imply \cA_p[T_\ell']=\cB_p[T_\ell'] \imply \cA_p=\cB_p.
\end{equation*}
Doing this successively, we have
\begin{equation*}
\cA_p[T_\ell : \ell \mid N_2]=\cB_p[T_\ell : \ell \mid N_2] \imply \cA_p=\cB_p.
\end{equation*}
This completes the proof.
\end{proof}

We finish the proof of Theorem \ref{theorem : main theorem} by induction.
\begin{proof}[Proof of Theorem \ref{theorem : main theorem}]
By the result in Section \ref{section: strategy}, we have
\begin{equation}\label{eqn: 00}
\scC(N')[p^\infty]=J_0(N')(\Q)_\tor[p^\infty],
\end{equation}
where $N'=\prod_{i=1}^u \ell_i$. Thus, the result follows when $N$ is squarefree.

Next, let $d=\prod_{i=1}^u \ell_i^{f_i}$ be a divisor of $N$ such that $1\leq f_i\leq r_i$ for all $1\leq i \leq u$. If $d$ is not squarefree (including $d=N$), then 
by \eqref{eqn22} we have
\begin{equation*}
\scC(d)[p^\infty, T_{\ell} : \ell^2 \mid d]=J_0(d)(\Q)_\tor[p^\infty, T_{\ell} : \ell^2 \mid d].
\end{equation*}
Thus, the result easily follows by induction on $n(d)$, where $n(d):=\sum_{j=t+1}^u (f_j-1)$. Indeed, if $n(d)=1$, then all the assumptions in Theorem \ref{thm: inductive method} are satisfied for $N=d$ as we have (\ref{eqn: 00}), and so 
\begin{equation*}
\scC(d)[p^\infty]=J_0(d)(\Q)_\tor[p^\infty].
\end{equation*}
Suppose that $n(d)=2$. Since $n(d/\ell)=1$ for any prime $\ell$ such that $\ell^2$ divides $d$, the assumption in Theorem \ref{thm: inductive method} holds. Hence 
we have 
\begin{equation*}
\scC(d)[p^\infty]=J_0(d)(\Q)_\tor[p^\infty].
\end{equation*}
This proves the result for all $d$ with $n(d)=2$. Doing this successively, we complete the proof.
\end{proof}

\begin{remark}
We would like to point out a possible generalization of the methods used, and obstacles in the generalization. Suppose that $p^2$ does divide $N$. If we could prove Theorems \ref{thm: index}, 
\ref{thm: multi one} and \ref{thm: inductive method}, then all the arguments above would work verbatim. Thus, it suffices to prove such generalizations.

First, Theorem \ref{thm: inductive method} is true under the additional assumption that $p^4$ does not divide $N$ thanks to Lemma \ref{lemma: image of beta}. If $p^4$ does divide $N$, we need another argument.

Next, to prove Theorem \ref{thm: index}  we should develop some duality between certain ``modular forms'' and the Hecke algebras (cf. Proposition \ref{prop: duality}).
Also, we should understand ``the difference'' between the order of $[C_0^\bde]$ and 
the residue of the Eisenstein series $E_0^\bde$ at a certain cusp. 
As we already studied (Lemmas \ref{lemma: order} and \ref{lemma: residue}), their ratio is $\prod_{j=t+1}^u \ell_j$ (up to $2$). This discrepancy is a real obstacle.

Finally, $J_0(N)$ does not have semistable reduction at $p$, i.e., 
the special fiber $\cJ_p$ of the N\'eron model does have a non-trivial unipotent part. 
If we could prove that the kernel of $\fm_0^\bde$ on this unipotent part of $\cJ_p$ is trivial, then all the arguments in Theorem \ref{thm: multi one} work verbatim. However, the author does not know this yet.
\end{remark}

\subsection{Acknowledgments}
The author would like to thank Kenneth Ribet for his inspired suggestions and comments. 
He is also grateful to Myungjun Yu for many suggestions toward the correction and improvement of this article.
Finally, he would like to thank the anonymous referee for numerous corrections, suggestions and valuable remarks.
This work was supported the Seoul National University Research Grant in 2022 and by National Research Foundation of Korea(NRF) grant funded by the Korea government(MSIT) (No. 2019R1C1C1007169 and No. 2020R1A5A1016126). 

\ms


\end{document}